\documentclass[a4paper, 11pt]{amsart}
\usepackage{amssymb}
\usepackage[T1]{fontenc}
\usepackage[english]{babel}
\usepackage{times}
\usepackage{a4wide}
\usepackage{color}
\usepackage{graphicx}
\usepackage{amsmath, amsthm}
\usepackage[normalem]{ulem}
\usepackage{comment}
\usepackage[numbers]{natbib}

\makeatletter
\def\section{\@startsection{section}{1}%
 \z@{.7\linespacing\@plus\linespacing}{.5\linespacing}%
 {\normalfont\scshape\Large\centering}}
\def\subsection{\@startsection{subsection}{2}%
 \z@{.5\linespacing\@plus.7\linespacing}{.5\linespacing}%
 {\normalfont\scshape\centering}}
\makeatother

\newcounter{licznik}[section]

\newtheorem{theorem}[licznik]{Theorem}
\newtheorem{Lemma}[licznik]{Lemma} 
\newtheorem{Proposition}[licznik]{Proposition}
\newtheorem{cor}[licznik]{Corollary}
\newtheorem{remark}[licznik]{Remark}
\newtheorem{definition}[licznik]{Definition}

\def\ud{d}
\def\tl{\tilde}
\def\ve{\varepsilon}

\def\limsup{\mathop{\lim\sup}}
\def\liminf{\mathop{\lim\inf}}
\def\esssup{\operatornamewithlimits{\mathrm{ess\,sup}}}
\def\essinf{\operatornamewithlimits{\mathrm{ess\,inf\vphantom{p}}}}

\def\ve{\varepsilon}
\def\ef{\mathcal{F}}
\def\ee{\mathbb{E}}
\def\er{\mathbb{R}}
\def\prob{\mathbb{P}}
\def\te{\mathcal{T}}

\def\laa{\mathcal{A}}
\def\laac{\mathcal{A}_{ac}}
\def\es{\mathcal{S}}
\def\inf{\operatornamewithlimits{inf\vphantom{p}}}
\newcommand{\ind}[1]{I_{#1}}
\newcommand{\indd}[1]{\ind{\{#1\}}}

\newcommand{\les}{\le}
\newcommand{\ges}{\ge}
\newcommand{\mcal}{\mathcal}

\newcommand{\eps}{\ve}
\newcommand{\oms}{\Omega^s}
\newcommand{\efs}{\ef^s}
\newcommand{\probs}{\prob^s}
\newcommand\mcalO{\mathcal{O}}

\newcommand{\what}{\widehat}

\makeatletter 
\def\punctuationSpace{\@ifnextchar.{}{\@ifnextchar,{}{\@ifnextchar;{}{ }}}}
\makeatother

\newcommand{\cadlag}{c\`adl\`ag\punctuationSpace}

\newcommand\as{\mbox{-a.s.}\punctuationSpace}
\renewcommand\ae{\mbox{-a.e.}\punctuationSpace}

\newcommand{\I}{I}
\newcommand{\J}{J}
\newcommand{\mcalI}{\mathcal{I}}
\newcommand{\mcalJ}{\mathcal{J}}
\newcommand{\mcalL}{{\mathcal{L}_b}}
\newcommand{\mcalG}{\mathcal{G}}
\newcommand{\mcalAcirc}{{\mathcal{A}^\circ}}
\newcommand{\mcalAcircAc}{{\mathcal{A}^\circ_{ac}}}
\newcommand{\tlmcalAcirc}{{\tilde{\mathcal{A}}^\circ}}

\definecolor{brightmaroon}{rgb}{0.7, 0.23, 0.2}
\definecolor{DarkOrange}{rgb}{0.8, 0.45, 0}

\makeatletter
\newcommand\assumptionlabel[1]{\hspace\labelsep
                               \normalfont\bfseries #1\ \ \gdef\@currentlabel{#1}}
\newenvironment{assumption}
               {\medskip\list{}{\labelwidth\z@ \itemindent-\leftmargin
                        }}
               {\endlist}
\makeatother

\begin{document}
\title[Non-Markovian Dynkin games with partial/asymmetric information]{On the value of non-Markovian Dynkin games \\ with partial and asymmetric information}
\author[T. De Angelis, N. Merkulov, J. Palczewski]{Tiziano De Angelis, Nikita Merkulov, Jan Palczewski}
\address{T.\ De Angelis: School of Management \& Economics, Dept.\ ESOMAS, University of Turin, C.so Unione Sovietica 218bis, 10134, Torino, Italy}
\email{tiziano.deangelis@unito.it (T. De Angelis)} 
\address{N.\ Merkulov and J.\ Palczewski: School of Mathematics, University of Leeds, LS2 9JT, Leeds, UK}
\email{mmnme@leeds.ac.uk (N. Merkulov)}
\email{j.palczewski@leeds.ac.uk (J. Palczewski)}
\thanks{{\em Acknowledgments}: T.~De Angelis gratefully acknowledges support by the EPSRC grant EP/R021201/1. We are grateful to an anonymous referee who pointed us to the existence of optimal strategies and suggested the equivalence of our topology with the one used by Baxter and Chacon \cite{BaxterChacon} and Meyer \cite{Meyer} (see our Lemma \ref{lem:top})}

\begin{abstract}
We prove that zero-sum Dynkin games in continuous time with partial and asymmetric information admit a value in randomised stopping times when the stopping payoffs of the players are general \cadlag measurable processes. As a by-product of our method of proof we also obtain existence of optimal strategies for both players. The main novelties are that we do not assume a Markovian nature of the game nor a particular structure of the information available to the players. This allows us to go beyond the variational methods (based on PDEs) developed in the literature on Dynkin games in continuous time with partial/asymmetric information. Instead, we focus on a probabilistic and functional analytic approach based on the general theory of stochastic processes and Sion's min-max theorem (M.\ Sion, Pacific J. Math., {\bf 8}, 1958, pp.\ 171-176). Our framework encompasses examples found in the literature on continuous time Dynkin games with asymmetric information and we provide counterexamples to show that our assumptions cannot be further relaxed.
\medskip

\end{abstract}

\maketitle

\section{Introduction}\label{sec:intro}

In this paper we develop a framework for the study of the existence of a value (also known as {\em Stackelberg equilibrium}) in zero-sum Dynkin games with partial/asymmetric information  in a non-Marko\-vian setting, when the payoffs are general \cadlag measurable  processes and players are allowed to use randomised stopping times. As a by-product of our method of proof we also obtain existence of optimal strategies for both players. The games are considered on both finite and infinite-time horizon and the horizon is denoted by $T$.  The payoff processes can be decomposed into the sum of a {\em regular} process (in the sense of Meyer \cite{Meyer}) and a pure jump process with mild restrictions on the direction of predictable jumps for one of the two players. Regular processes form a very broad class encompassing, for example, all \cadlag processes that are also quasi left-continuous (i.e., left-continuous over stopping times). 

We allow for a very general structure of the information available to the players. All processes are adapted to an overarching filtration $(\ef_t)$ whereas each player makes decisions based on her own filtration, representing her access to information. Letting $(\ef^{\,i}_t)$ be the filtration of the $i$-th player, with $i=1,2$, we only need to assume that $\ef^{\,i}_t\subseteq\ef_t$ for all $t\in[0,T]$. In particular, we cover the case in which players are equally (partially) informed, i.e., $\ef^1_t=\ef^2_t$, and, more importantly, the case in which they have asymmetric (partial) information, i.e., $\ef^1_t\neq\ef^2_t$.

Under this generality we prove that Dynkin games with second-mover advantage admit a value in mixed strategies (which in this context are represented by randomised stopping times) and optimal strategies exist for both players. 

Our framework encompasses most (virtually all) examples of zero-sum Dynkin games (in continuous time) with partial/asymmetric information that we could find in the literature (see, e.g., De Angelis et al. \cite{DGV2017}, \cite{DEG2020}, Gensbittel and Gr\"un \cite{GenGrun2019}, Gr\"un  \cite{Grun2013} and Lempa and Matom\"aki \cite{lempa2013}) and we give a detailed account of this fact in Section \ref{sec:examples} (notice that \cite{DGV2017} and \cite{lempa2013} obtain a saddle point for the game in pure strategies, i.e., using stopping times, but in very special examples). Broadly speaking, all those papers' solution methods hinge on variational inequalities and PDEs and share two key features: (i) a specific structure of the information flow in the game and (ii) the Markovianity assumption. In our work instead we are able to analyse games at a more abstract level that allows us to drop the Markovianity assumption and to avoid specifying an information structure. Of course the cost to pay for such greater generality is a lack of explicit results concerning the value and the optimal strategies (beyond their existence), which instead may be obtained in some problems satisfying (i) and (ii) above. We also show by several counterexamples that our main assumptions cannot be further relaxed as otherwise a value for the game may no longer exist.

Our methodology draws on the idea presented in Touzi and Vieille \cite{TouziVieille2002} of using Sion's min-max theorem (Sion \cite{Sion1958}). In \cite{TouziVieille2002} the authors are interested in non-Markovian zero-sum Dynkin games with full information in which first- and second-mover advantage may occur at different points in time, depending on the stochastic dynamics of the underlying payoff processes. In that context randomisation is essentially used by the players to attain a value in the game and avoid stopping simultaneously (another general study on such class of problems, but without using Sion's theorem, is contained in Laraki and Solan \cite{LarakiSolan2005}). Since our set-up is different, due to the partial/asymmetric information features and relaxed assumptions on the payoff processes, we encounter some non-trivial technical difficulties in repeating arguments from \cite{TouziVieille2002}; for example, our class of randomised stopping times is not closed with respect to the topology used in \cite{TouziVieille2002} (see Remark \ref{rem-TV-norm-doesnt-work}). For this reason we develop an alternative approach based on the general theory of stochastic processes combined with ideas from functional analysis.

\subsection{Literature review}
Existence of a value (Stackelberg equilibrium) in zero-sum Dynkin games is a research question that goes back to the 70's in the classical set-up where players have full and symmetric information. A comprehensive and informative review of the main results since Dynkin's inception of stopping games \cite{dynkin1969} is contained in the survey paper by Kifer \cite{kifer2013}. Here we recall that early results on the existence of a value in a diffusive set-up were obtained by Bensoussan and Friedman \cite{bensoussan1974} via PDE methods, and by Bismut \cite{bismut1977} via probabilistic methods (and allowing for processes with jumps). Those results were later extended to right-continuous Markov processes by Stettner \cite{stettner1982} (see also Stettner \cite{stettner1982b} and \cite{stettner1984}). In the non-Markovian setting the early results are due to Lepeltier and Maingueneau \cite{lepeltier1984}. Around the year 2000 zero-sum Dynkin games gained popularity thanks to their applications in mathematical finance suggested by Kifer \cite{kifer2000} (see also Kyprianou \cite{kyprianou2004} for an early contribution). Numerous other papers have addressed related questions, including the existence of value and optimal strategies, with various methods. The interested reader may consult Ekstr\"om and Peskir \cite{ekstrom2008} for modern results in a general Markovian setting or Ekstr\"om and Villeneuve \cite{ekstrom2006} for the special case of one-dimensional linear diffusions. It is also worth mentioning that nonzero-sum Dynkin games have been studied by many authors including, e.g., Hamadene and Zhang \cite{hamadene2010} for non-Markovian games,  Attard \cite{attard2018} for general Markovian games and De Angelis et al. \cite{de2018nash} for games on one-dimensional linear diffusions. All the papers mentioned in this (largely incomplete) literature review deal with players holding full and symmetric information and the value is found in {\em pure strategies}, i.e., in stopping times. 

Yasuda \cite{yasuda1985} was probably the first to study the existence of a value for Dynkin games with {\em randomised stopping times}, in the case of discrete-time Markov processes. In that context, randomisation is specified by assigning a probability of stopping at each time $n=1,2,\ldots$. A similar type of games for Markov chains with an absorbing state was also studied by Domansky \cite{domansky2002}. Rosenberg, Solan and Vieille \cite{rosenberg2001} developed more general results by removing the Markovianity assumption and the assumption of an ordering of payoffs (i.e., they did not require that there be a first- or second-mover advantage). Randomised stopping times are also used in mathematical finance, in the context of pricing and hedging game options in discrete time with transaction costs (see \cite[Sec.\ 5]{kifer2013}), and in mathematical economics, to construct subgame-perfect equilibria in (nonzero-sum) Dynkin games (see, e.g., Riedel and Steg \cite{riedel2017}). In continuous time, the most recent results on the existence of a value for non-Markovian zero-sum Dynkin games with randomised stopping times are contained in Touzi and Vieille \cite{TouziVieille2002} and in Laraki and Solan \cite{LarakiSolan2005}.

We emphasise that in all the papers mentioned above players are equally informed and the need for randomisation stems mainly from a specific structure of the game's payoff (often due to the lack of an ordering of the payoff processes that induces alternating first- and second-mover advantage). Our work instead is inspired by a more recent strand of the literature on continuous-time Dynkin games that addresses the role of information across players and its impact on their strategies (see also Section \ref{sec:examples} for a fuller account). We believe this strand was initiated with work by Gr\"un \cite{Grun2013}, where one of the players knows the payoff process in the game while the other one has access only to an initial distribution of possible payoff processes. Gr\"un proves existence of a value in randomised strategies and the existence of an optimal strategy for the informed player. In Gr\"un and Gensbittel \cite{GenGrun2019} players observe two different stochastic processes and the payoffs in the game depend on both processes. The authors prove existence of a value and, under some additional conditions, of optimal strategies for both players. Both these papers attack the problem via a characterisation of the value as the viscosity solution of a certain variational inequality (of a type which is rather new in the literature and is inspired by similar results in the context of differential games; see, e.g., Cardaliaguet and Rainer \cite{cardaliaguet2009}). Free-boundary methods in connection with randomised stopping times are instead used in De Angelis et al. \cite{DEG2020}, where players have asymmetric information regarding the drift of a linear diffusion underlying the game, and in Ekstr\"om et al. \cite{ekstrom2017}, where the two players estimate the drift parameter according to two different models. The methods used in those papers cannot be extended to the non-Markovian framework of our paper.

Although not directly related to Dynkin games, we notice that methods from functional analysis and the general theory of stochastic processes have been used recently to study optimal stopping problems by Pennanen and Perkk\"io \cite{pennanen2018}. By relaxing the problem to include randomised stopping times the authors reduce the optimal stopping problem to an optimisation of a linear functional over a convex set of randomised stopping times and find that the solution exists as an extreme point, i.e., a pure stopping time. Closely related contributions date back to Baxter and Chacon \cite{BaxterChacon} and Meyer \cite{Meyer} who establish compactness of the set of randomised stopping times in weak topologies defined by functionals which can be interpreted as stopping of quasi left-continuous processes, in \cite{BaxterChacon}, and regular processes, in \cite{Meyer}. In our game framework we need to rely on min-max arguments instead of convex optimisation as in the optimal stopping case, so compactness arguments are not immediately applicable. However, \cite{BaxterChacon, Meyer} inspired our approach and some of our convergence results in Section \ref{sec:tech}. Furthermore, the topology of \cite{Meyer} on the set of randomised stopping times turns out to be equivalent to the topology obtained via different routes in our paper (Lemma \ref{lem:top}).

\subsection{Structure of the paper}
The paper is organised as follows. The problem is set in Section \ref{sec:setting} where we also state our main result on the existence of a value in full generality (Theorem \ref{thm:main2}). For the ease of readability we also state a version of the result under slightly stronger conditions on the underlying processes (Theorem \ref{thm:main}), which allows a more linear approach to the proof. An extension to the case of a game with conditioning on some initial information is also stated as Theorem \ref{thm:ef_0_value}. Before turning to proofs of our results we use Section \ref{sec:examples} to illustrate how our framework encompasses Dynkin games in continuous time with partial and asymmetric information that have appeared in the literature to date. Section \ref{sec:reform} is used to reformulate the Dynkin game in terms of a game of increasing (singular) controls. Section \ref{sec-Sion-existence-of-value} begins with a statement of Sion's min-max theorem which is followed by a (short) proof of our Theorem \ref{thm:main}. The latter is based on several technical results which are addressed in detail in Sections \ref{sec:tech}, \ref{sec:verif} and \ref{sec:approx}. The proof of Theorem \ref{thm:main2} is then given in Section \ref{sec:relax} and the one of Theorem \ref{thm:ef_0_value} is finally given in Section \ref{sec:ef_functional}. We close the paper with a few counterexamples in Section \ref{sec:Nikita-examples} showing that our conditions cannot be relaxed.

\section{Problem setting and main results}\label{sec:setting}
Fix a complete probability space $(\Omega,\ef, \prob)$ equipped with a filtration $(\ef_t)_{t\in[0,T]}$, where $T\in(0,\infty]$ is the time horizon of our problem. All random variables, processes and stopping times are considered on this filtered probability space unless specified otherwise. We write $\ee$ for the expectation with respect to measure $\prob$. By a \emph{measurable process} we mean a stochastic process which is $\mcal{B}([0,T])\times\mathcal F$-measurable. We denote by $\mcalL$ a Banach space of \cadlag measurable processes with the norm
\[
\| X \|_{\mcalL} := \ee \Big[\sup_{t\in[0,T]} |X_t|\Big] < \infty.
\]
A process $(X_t)_{t\in [0,T]} \in \mcalL$ is called \emph{regular} if 
\begin{align}\label{eq:cond-reg}
\ee [X_\eta - X_{\eta-}|\ef_{\eta-}] = 0\quad \prob\as \text{ for all predictable $(\ef_t)$-stopping times $\eta$.} 
\end{align}
Notice that if $T=\infty$, then $\infty$ is a one-point compactification of $[0, \infty)$, so that \cadlag and regular processes are understood as follows (c.f. \cite[Remark VI.53e]{DellacherieMeyer}): a process $(X_t)_{t \in [0, \infty]}$ is \cadlag if it is \cadlag on $[0, \infty)$ and the limit $X_{\infty-}:=\lim_{t \to \infty} X_t$ exists; the random variable $X_\infty$ if $\ef_\infty$-measurable and $\ef_\infty$ is potentially different from $\ef_{\infty-} = \sigma\big(\cup_{t \in [0, \infty)} \ef_t\big)$. Furthermore, $(X_t)_{t\in[0,\infty]}$ is regular, if it is regular on $[0, \infty)$ and 
\[
\ee[X_\infty - \lim_{t \to \infty} X_t | \ef_{\infty-} ] = 0.
\]
Throughout the paper we consider several filtrations and stochastic processes, so in order to keep the notation simple we will often use $(\ef_t)$ instead of $(\ef_t)_{t\in[0,T]}$ and similarly $(X_t)$ (or simply $X$) for a process $(X_t)_{t\in[0,T]}$.

We consider two-player zero-sum Dynkin games on the (possibly infinite) horizon $T$. Actions of the first player are based on the information contained in a filtration $(\ef^1_t) \subseteq (\ef_t)$. Actions of the second player are based on the information contained in a filtration $(\ef^2_t) \subseteq (\ef_t)$. Each player selects a random time (taking values in $[0,T]$) based on the information she acquires via her filtration: the first player's random time is denoted by $\tau$ while the second player's random time is $\sigma$. The game terminates at time $\tau \wedge \sigma \in[0, T]$ with the first player delivering to the second player the payoff
\begin{equation}\label{eqn:payoff}
\mcal{P} (\tau, \sigma) = f_{\tau} \ind{\{\tau<\sigma\}} 
+
g_{\sigma} \ind{\{{\sigma}<{\tau}\}}
+
h_{\tau} \ind{\{\tau=\sigma\}}.
\end{equation}
The first player (or $\tau$-player) is the {\em minimiser} in the game whereas the second player (or $\sigma$-player) is the {\em maximiser}. That means that the former will try to minimise the expected payoff (see \eqref{eq-uninf-payoff} below) while the latter will try to maximise it.

The {\em payoff processes} $f$, $g$ and $h$ satisfy the following conditions:
\begin{assumption}
\item[(A1)]\label{eq-integrability-cond} $f, g \in\mcalL$, 
\item[(A2)]\label{ass:regular} $f, g$ are $(\ef_t)$-adapted regular processes,
\item[(A3)]\label{eq-order-cond} $f_t\ge h_t \ge g_t$  for all $t\in[0,T]$, $\prob$\as,
\item[(A4)]\label{eq-terminal-time-order-cond} $h$ is an $(\ef_t)$-adapted, measurable process.
\end{assumption}
In particular, we do not assume that $h$ is \cadlag.

In the context of zero-sum games that we will address below, our assumption \ref{eq-order-cond} corresponds to the so-called games with the {\em second-mover advantage}, i.e., games in which both players have an incentive to wait for the opponent to end the game.

Assumption \ref{eq-integrability-cond} is natural in the framework of optimal stopping problems (see, e.g., \cite[Section 2.2]{Peskir2006}, \cite[Eq. (D.29)]{Karatzas1998}) and Dynkin games (\cite{TouziVieille2002}). With \ref{ass:regular} we replace semimartingale assumptions on $f$ and $g$ from \cite{TouziVieille2002} while imposing the regularity condition dating back to \cite{Meyer} in the optimal stopping framework. Regular processes encompass a large family of stochastic processes encountered in applications. It is straightforward to see that quasi left-continuous processes \cite[Section III.11]{RogersWilliams} are regular. In the Markovian framework all standard processes \cite[Def. I.9.2]{Blumenthal1968} and, in particular, weak Feller processes (\cite[Def. III.6.5 and Thm. III.11.1]{RogersWilliams} or \cite[Thm. I.9.4]{Blumenthal1968}) are regular. Hence, strong and weak solutions to stochastic differential equations (SDEs) driven by multi-dimensional Brownian motion (or L\`evy process and bounded coefficients) \cite[Section 6.7]{Appelbaum2009} and solutions to jump-diffusion SDEs are also regular processes. More generally, a regular process is allowed to jump at predictable times $\eta$, provided that such jumps have zero mean conditional on $\ef_{\eta-}$. 

We subsequently relax Assumption \ref{ass:regular} by allowing payoff processes with predictable jumps of nonzero mean. That is, we replace \ref{ass:regular} with the following:
\begin{assumption}
 \item[(A2')] \label{ass:regular_gen} Processes $f$ and $g$ have the decomposition $f = \tl f + \hat f$, $g = \tl g + \hat g$ with 
 \begin{enumerate}
  \item $\tl f, \tl g \in\mcalL$,
  \item $\tl f, \tl g $ are $(\ef_t)$-adapted regular processes,
  \item $\hat f, \hat g$ are $(\ef_t)$-adapted (right-continuous) piecewise-constant processes of integrable variation with $\hat f_{0} = \hat g_0 = 0$, $\Delta \hat f_T = \hat f_{T} - \hat f_{T-} = 0$ and $\Delta \hat g_T =\hat g_T- \hat g_{T-}= 0$,
  \item either $\hat f$ is non-increasing or $\hat g$ is non-decreasing.
 \end{enumerate}
\end{assumption}
Notice that there are non-decreasing processes $(\hat f^+_t), (\hat f^-_t), (\hat g^+_t), (\hat g^-_t) \in \mcalL$ starting from $0$ such that $\hat f = \hat f^+-\hat f^-$ and $\hat g = \hat g^+ - \hat g^-$ \cite[p. 115]{DellacherieMeyer}.

Under Assumption \ref{ass:regular_gen}, we allow jumps of $\hat f$ in any direction and only upward jumps of $\hat g$, or, viceversa, jumps of $\hat g$ in any direction and downward jumps of $\hat f$. This ensures a certain closedness property (see Section \ref{sec:relax}).
It is worth emphasising that further relaxation of condition \ref{ass:regular_gen} is not possible in the generality of our setting as demonstrated in Remark \ref{rem:contrad} and in Section \ref{subsec:example_jumps}. While regular processes have no restrictions on non-predictable jumps, \ref{ass:regular_gen} relaxes condition in Eq. \eqref{eq:cond-reg} by allowing predictable jumps with non-zero (conditional) mean. The necessity to restrict the direction of predictable jumps of one of the payoff processes is a new feature introduced by the asymmetry of information. In classical Dynkin games it is not necessary, see \cite{ekstrom2008, lepeltier1984, stettner1982b}.

We further require a technical assumption
\begin{assumption}
\item[(A5)]\label{ass:filtration} The filtrations $(\ef_t)$ and $(\ef^i_t)$, $i=1,2$, satisfy the usual conditions, i.e., they are right-continuous and $\ef^i_0$, $i=1,2$, contain all sets of $\prob$-measure zero.
\end{assumption}

Players assess the game by looking at the \emph{expected payoff}
\begin{equation} 
N(\tau,\sigma)= \ee \big[ \mcal{P} (\tau, \sigma) \big].
\label{eq-uninf-payoff}
\end{equation}
The game is said to have a \emph{value} if
\[
\sup_\sigma \inf_\tau N(\tau, \sigma) = \inf_\tau \sup_\sigma  N(\tau, \sigma),
\]
where, for now, we do not specify the nature of the admissible random times $(\tau,\sigma)$. The mathematical difficulty with establishing existence of a value lies in the possibility to swap the order of `inf' and `sup' and this is closely linked to the choice of admissible random times. Furthermore, an admissible pair $(\tau_*,\sigma_*)$ is said to be a {\em saddle point} (or a pair of \emph{optimal strategies}) if
\[
N(\tau_*,\sigma)\le N(\tau_*,\sigma_*)\le N(\tau,\sigma_*),
\]
for all other admissible pairs $(\tau,\sigma)$.

\begin{remark}\label{rem:ineq}
Our problem formulation enjoys a symmetry which will be later used in proofs. Since we do not make assumptions on the sign of $f$, $g$, $h$, if the value exists for the game with payoff $\mcal{P}(\tau,\sigma)$, then it also exists for the game with payoff $\mcal{P}'(\tau,\sigma):=-\mcal{P}(\tau,\sigma)$. However, in the latter game the $\tau$-player is a maximiser and the $\sigma$-player is a minimiser, by the simple fact
\begin{align}\label{eq:swap}
\sup_\sigma\inf_\tau \ee[\mcal{P}(\tau,\sigma)]=-\inf_\sigma\sup_\tau \ee[\mcal{P}'(\tau,\sigma)],
\end{align}
where, obviously, in $\mcal{P}'$ we have the payoff processes $f'_t:=-f_t$, $g'_t:=-g_t$ and $h'_t:=-h_t$.
\end{remark}

It has been indicated in the literature that games with asymmetric information may not have a value if players' strategies are stopping times for their respective filtrations, see \cite[Section 2.1]{Grun2013}. Indeed, in Section \ref{sec:Nikita-examples} we demonstrate that the game studied in this paper may not, in general, have a value if the  first player (resp. the second player) uses $(\ef^1_t)$-stopping times (resp. $(\ef^2_t)$-stopping times). It has been proven in certain Markovian set-ups that the relaxation of player controls to randomised stopping times may be sufficient for the existence of the value (see, e.g., \cite{Grun2013}, \cite{GenGrun2019}). The goal of this paper is to show that this is indeed true in the generality of our non-Markovian set-up for the game with payoff \eqref{eq-uninf-payoff}.

The framework of this paper encompasses all two-player zero-sum Dynkin games in continuous time that we found in the literature. Indeed, when $(\ef_t) = (\ef^1_t) = (\ef^2_t)$, the game \eqref{eq-uninf-payoff} is the classical Dynkin game with full information for both players. The case of $(\ef^1_t) = (\ef^2_t)$ but $(\ef^1_t) \ne (\ef_t)$ corresponds to a game with partial but symmetric information about the payoff processes (e.g., \cite{DGV2017}), whereas $(\ef^1_t) \ne (\ef^2_t)$ is the game with asymmetric information. One can have $(\ef^1_t)=(\ef_t)$, i.e., only the second player is uninformed (e.g., \cite{Grun2013}), or $(\ef^1_t)\ne(\ef_t)$ and $(\ef^2_t) \ne (\ef_t)$, i.e., both players access different information flows and neither of them has full knowledge of the underlying world (e.g., \cite{GenGrun2019}). In Section \ref{sec:examples}, we present in full detail how games with asymmetric information studied in the literature fit into our framework.

As mentioned above the concept of randomised stopping time is central in our work, so we introduce it here. For that we need to consider increasing processes: given a filtration $(\mcal{G}_t) \subseteq (\ef_t)$ let
\begin{align*}
\mcalAcirc (\mcalG_t):=&\,\{\rho\,:\,\text{$\rho$ is $(\mcalG_t)$-adapted with $t\mapsto\rho_t(\omega)$ \cadlag,}\\
&\qquad\,\text{non-decreasing, $\rho_{0-}(\omega)=0$ and $\rho_T(\omega)=1$ for all $\omega\in\Omega$}\}.
\end{align*}
In the definition of $\mcalAcirc(\mcalG_t)$ we take the opportunity to require that the stated properties hold for all $\omega \in \Omega$. This leads to no loss of generality if $\mcalG_0$ contains all $\prob$-null sets of $\Omega$. Hence for any $\omega\in \mathcal{N} \subset \Omega$ with $\mathbb{P}(\mathcal{N})=0$ we can simply set $\rho_t(\omega)=0$ for $t\in[0,T)$ and $\rho_T(\omega)=1$. Recall that in the infinite-time horizon case, $T=\infty$, we understand $\rho_T$ as an $\ef_{\infty}$-measurable random variable while $\rho_{T-}:=\lim_{t\to\infty}\rho_t$ (which exists by the assumption that $(\rho_t)$ is a \cadlag process).
Randomised stopping times can be defined as follows.
\begin{definition}\label{def-rand-st}
Given a filtration $(\mcal{G}_t) \subseteq (\ef_t)$, a random variable $\eta$ is called a \emph{$(\mcal{G}_t)$-randomised stopping time} if there exists a random variable $Z$ with uniform distribution $U([0,1])$, independent of $\ef_T$, and a process $\rho\in\mcalAcirc(\mcalG_t)$ 
such that 
\begin{equation}
\eta=\eta(\rho,Z)=\inf\{t\in [0,T]: \rho_t > Z\}, \quad \prob\as
\label{eq-def-rand-st}
\end{equation}
The variable $Z$ is called a \emph{randomisation device} for the randomised stopping time $\eta$, and the process $\rho$ is called the \emph{generating process}. The set of $(\mcal{G}_t)$-randomised stopping times is denoted by $\te^R(\mcal{G}_t)$. It is assumed that randomisation devices of different stopping times are independent.
\end{definition}
We refer to \cite{Solanetal2012}, \cite{TouziVieille2002} for an extensive discussion on various definitions of randomised stopping times and conditions that are necessary for their equivalence. To avoid unnecessary complication of notation, we assume that the probability space $(\Omega, \ef, \prob)$ supports two independent random variables $Z_\tau$ and $Z_\sigma$ which are also independent of $\ef_T$ and are the randomisation devices for the randomised stopping times $\tau$ and $\sigma$ of the two players.  

\begin{definition}\label{def-value-rand-strat}
Define
\begin{equation*}
V_*:=\sup_{\sigma \in \te^R(\ef^2_t)} \inf_{\tau\in \te^R(\ef^1_t)} N(\tau,\sigma)\quad\text{and}\quad V^*:= \inf_{\tau\in \te^R(\ef^1_t)}\sup_{\sigma \in \te^R(\ef^2_t)}  N(\tau,\sigma).
\end{equation*}
The \emph{lower value} and {\em upper value of the game in randomised strategies} are given by $V_*$ and $V^*$, respectively.  If they coincide, the game is said to have a \emph{value in randomised strategies} $V=V_*=V^*$. 
\end{definition}
The following theorem states the main result of this paper.
\begin{theorem}
\label{thm:main2}
Under assumptions \ref{eq-integrability-cond}, \ref{ass:regular_gen}, \ref{eq-order-cond}-\ref{ass:filtration}, the game has a value in randomised strategies. Moreover, if $\hat f$ and $\hat g$ in \ref{ass:regular_gen} are non-increasing and non-decreasing, respectively, there exists a pair $(\tau_*,\sigma_*)$ of optimal strategies.
\end{theorem}
For the clarity of presentation of our methodology, we first prove a theorem with more restrictive regularity properties of payoff processes and then show how to extend the proof to the general case of Theorem \ref{thm:main2}.
\begin{theorem}
\label{thm:main}
Under assumptions \ref{eq-integrability-cond}-\ref{ass:filtration}, the game has a value in randomised strategies and there exists a pair $(\tau_*,\sigma_*)$ of optimal strategies.
\end{theorem}
Proofs of the above theorems are given in Section \ref{sec:sions}. They rely on two key results: an approximation procedure (Propositions \ref{thm:conv_lipsch} and \ref{thm:conv_lipsch_gen}) and an auxiliary game with `nice' regularity properties (Theorem \ref{th-value-cont-strat} and \ref{th-value-cont-strat_gen}) which enables the use of a known min-max theorem (Theorem \ref{th-the-Sion}).

The $\sigma$-algebra $\ef_0$ is not assumed to be trivial. It is therefore natural to consider a game in which players assess their strategies ex-post, i.e., after observing the information available to them at time $0$ when their first action may take place. Allowing for more generality, let $\mcalG$ be a $\sigma$-algebra contained in $\ef^1_0$ and in $\ef^2_0$, i.e., containing only information available to both players at time $0$. The expected payoff of the game in this case takes the form (recall that $\tau,\sigma\in[0,T]$):
\begin{equation}\label{eqn:cond_func}
\ee\big[ \mcal{P}(\tau, \sigma) \big| \mcalG \big]
=
\ee\big[ f_{\tau} \ind{\{\tau<\sigma\}} 
+
g_{\sigma} \ind{\{{\sigma}<{\tau}\}}
+
h_{\tau} \ind{\{\tau=\sigma\}}\big| \mcalG \big].
\end{equation}

The proof of the following theorem is in Section \ref{sec:ef_functional}.
\begin{theorem}\label{thm:ef_0_value}
Under the assumptions of Theorem \ref{thm:main2} and for any $\sigma$-algebra $\mcalG \subseteq \ef^1_0 \cap \ef^2_0$, the $\mcalG$-conditioned game has a value, i.e.
\begin{equation}\label{eqn:value_ef0}
\esssup_{\sigma \in \te^R(\ef^2_t)} \essinf_{\tau\in \te^R(\ef^1_t)} \ee\big[ \mcal{P}(\tau, \sigma) \big| \mcalG \big] = \essinf_{\tau\in \te^R(\ef^1_t)}\esssup_{\sigma \in \te^R(\ef^2_t)}  \ee\big[\mcal{P}(\tau, \sigma) \big| \mcalG \big], \qquad \prob\as
\end{equation}
Moreover, if $\hat f$ and $\hat g$ in \ref{ass:regular_gen} are non-increasing and non-decreasing, respectively, there exists a pair $(\tau_*,\sigma_*)$ of optimal strategies in the sense that
\begin{equation}\label{eqn:saddleG}
\ee\big[ \mcal{P}(\tau_*, \sigma) \big| \mcalG \big]
\le
\ee\big[ \mcal{P}(\tau_*, \sigma_*) \big| \mcalG \big]
\le
\ee\big[ \mcal{P}(\tau, \sigma_*) \big| \mcalG \big], \qquad \prob\as
\end{equation}
for all other admissible pairs $(\tau,\sigma)$.

\end{theorem}

\section{Examples}\label{sec:examples}
Before moving on to prove the theorems stated above, in this section we illustrate some of the specific games for which our general results apply. We draw form the existing literature on two-player zero-sum Dynkin games in continuous time and show that a broad class of these (all those we are aware of) fits within our framework. Since our contribution is mainly to the theory of games with partial/asymmetric information, we exclude the well-known case of games with full information which has been extensively studied (see our literature review in the introduction). 

\subsection{Game with partially observed scenarios}\label{subsec:game_1}
Our first example extends the setting of \cite{Grun2013} and it reduces to that case if $J=1$ and the {\em payoff processes} $f$, $g$ and $h$ are deterministic functions of an It\^o diffusion $(X_t)$ on $\mathbb{R}^d$, i.e., $f_t=f(t,X_t)$, $g_t=g(t,X_t)$ and $h_t=h(t, X_t)$. On a discrete probability space $(\oms, \efs, \probs)$, consider two random variables $\mcalI$ and $\mcalJ$ taking values in $\{1,\ldots,\I\}$ and in $\{1,\ldots,\J\}$, respectively. Denote their joint distribution by $(\pi_{i,j})_{i=1, \ldots, \I,j=1,\ldots,\J}$ so that $\pi_{i,j} = \probs(\mcalI = i,\mathcal J=j)$. The indices $(i,j)$ are used to identify the {\em scenario} in which the game is played and are the key ingredient to model the asymmetric information feature. Consider another probability space $(\Omega^p, \ef^p, \prob^p)$ with a filtration $(\ef^p_t)$ satisfying the usual conditions, and $(\ef^p_t)$-adapted payoff processes $f^{i,j}$, $g^{i,j}$, $h^{i,j}$, with $(i,j)$ taking values in $\{1,\ldots,\I\} \times \{1,\ldots,\J\}$. For all $i,j$, we assume that $f^{i,j}$, $g^{i,j}$, $h^{i,j}$ satisfy conditions \ref{eq-integrability-cond}-\ref{eq-terminal-time-order-cond}.

The game is set on the probability space $(\Omega, \ef, \prob) := (\Omega^p \times \oms, \ef^p \vee \efs, \prob^p \otimes \probs)$. The first player is informed about the outcome of $\mcalI$ before the game starts but never directly observes $\mcalJ$. Hence, her actions are adapted to the filtration $\ef^1_t = \ef^p_t \vee \sigma(\mcalI)$. Conversely, the second player knows $\mcalJ$ but not $\mcalI$, so her actions are adapted to the filtration $\ef^2_t = \ef^p_t \vee \sigma(\mcalJ)$. Given a choice of random times $\tau\in \te^R(\ef^1_t)$ and $\sigma \in \te^R(\ef^2_t)$ for the first and the second player, the payoff is
\begin{equation*}
\mcal{P} (\tau, \sigma) = f^{\mcalI,\mcalJ}_{\tau} \ind{\{\tau<\sigma\}} 
+
g^{\mcalI,\mcalJ}_{\sigma} \ind{\{{\sigma}<{\tau}\}}
+
h^{\mcalI,\mcalJ}_\tau \ind{\{\tau = \sigma\}}.
\end{equation*}
Players assess the game by looking at the expected payoff as in \eqref{eq-uninf-payoff}. It is worth noticing that this corresponds to the so-called `{\em ex-ante}' expected payoff, i.e., the expected payoff before the players acquire the additional information about the values of $\mcalI$ and $\mcalJ$. The structure of the game is common knowledge, i.e., both players know all processes $f^{i,j}$, $g^{i,j}$ and $h^{i,j}$ involved; however, they have partial and asymmetric knowledge on the couple $(i,j)$ which is drawn at the start of the game from the distribution of $(\mcalI,\mcalJ)$.

Drawing a precise parallel with the framework of Section \ref{sec:setting}, the above setting corresponds to $f_t = f^{\mcalI,\mcalJ}_t$, $g_t = g^{\mcalI,\mcalJ}_t$, and $h_t = h_t^{\mcalI,\mcalJ}$  with the filtration $\ef_t = \ef^p_t \vee \sigma(\mcalI, \mcalJ)$. The observation flows for the players are given by $(\ef^1_t)$ and $(\ef^2_t)$, respectively. 

The particular structure of players' filtrations $(\ef^1_t)$ and $(\ef^2_t)$ allows for the following decomposition of randomised stopping times, see \cite[Proposition 3.3]{esmaeeli2018} (recall the radomisation devices $Z_\tau\sim U([0,1])$ and $Z_\sigma\sim U([0,1])$, which are mutually independent and independent of $\ef_T$). 
\begin{Lemma}\label{lem:tau_decomposition}
Any $\tau \in \te^R(\ef^1_t)$ has a representation
\begin{equation}\label{eqn:tau_decomposition}
\tau = \sum_{i=1}^\I \ind{\{\mcalI = i\}} \tau_i, 
\end{equation}
where $\tau_1,\ldots,\tau_\I \in \te^R(\ef^p_t)$, with generating processes $\xi^1,\ldots,\xi^\I \in \mcalAcirc (\ef^p_t)$ and a common randomisation device $Z_\tau$.
An analogous representation holds for $\sigma$ with $\sigma_1, \ldots, \sigma_\J \in \te^R(\ef^p_t)$, generating processes $\zeta^1_t, \ldots, \zeta^\J_t \in \mcalAcirc (\ef^p_t)$, and a common randomisation device $Z_\sigma$. 
\end{Lemma}
\begin{cor}
Any $(\ef^1_t)$-stopping time $\tau$ has a decomposition \eqref{eqn:tau_decomposition} with $\tau_1,\ldots,\tau_\I$ being $(\ef^p_t)$-stopping times (and analogously for $(\ef^2_t)$-stopping times).
\end{cor}
Hence, given a realisation of the idiosyncratic scenario variable $\mcalI$ (resp.\ $\mcalJ$), the first (second) player chooses a randomised stopping time whose generating process is adapted to the common filtration $(\ef^p_t)$. The resulting expected payoff can be written as
\begin{equation*}
N(\tau, \sigma) = \sum_{i=1}^\I \sum_{j=1}^\J \pi_{i,j} \ee \Big[ f^{i,j}_{\tau_i} \ind{\{\tau_i<\sigma_j\}}+
g^{i,j}_{\sigma_j} \ind{\{{\sigma_j}<{\tau_i}\}}+ h^{i,j}_{\tau_i} \ind{\{\tau_i = \sigma_j\}} \Big].
\end{equation*}

\subsection{Game with a single partially observed dynamics} \label{subsec:game_2}
Our second example generalises the set-ups of \cite{DGV2017} and \cite{DEG2020} and reduces to those cases when $J=2$, the time horizon is infinite and the payoff processes are (particular) time-homogeneous functions of a (particular) one-dimensional diffusion. Here the underlying dynamics of the game is a diffusion, whose drift depends on the realisation of an independent random variable $\mcalJ\in\{1,\ldots, J\}$. Formally, on a probability space $(\Omega, \ef, \prob)$ we have a Brownian motion $(W_t)$ on $\er^d$, an independent random variable $\mcalJ\in\{1,\ldots, J\}$ with distribution $\pi_j=\prob(\mcalJ=j)$, and a process $(X_t)$ on $\er^d$ with the dynamics
\[
dX_t=\sum_{j=1}^J \ind{\{\mcalJ=j\}} \mu_j(X_t)dt+\sigma(X_t)dW_t,\quad X_0=x,
\]
where $\sigma$, $(\mu_j)_{j=1,\ldots J}$ are given functions (known to both players) that guarantee existence of a unique strong solution of the SDE for each $j=1,\ldots J$. The payoff processes are deterministic functions of the underlying process, i.e., $f_t=f(t,X_t)$, $g_t=g(t,X_t)$ and $h_t=h(t,X_t)$, and they are known to both players. We assume that the payoff processes satisfy conditions \ref{eq-integrability-cond}-\ref{eq-terminal-time-order-cond}. It is worth to remark that in the specific setting of \cite{DGV2017} the norms $\| f \|_{\mcalL}$ and $\| g \|_{\mcalL}$ are not finite so that our results cannot be directly applied. However, the overall structure of the game in \cite{DGV2017} is easier than ours so that some other special features of the payoff processes can be used to determine existence of the value therein.

To draw a precise parallel with the notation from Section \ref{sec:setting}, here we take $\ef_t=\ef^W_t\vee\sigma(\mcalJ)$, where $(\ef^W_t)$ is the filtration generated by the Brownian sample paths and augmented with $\prob$-null sets. Both players observe the dynamics of $X$, however they have partial/asymmetric information on the value of $\mcalJ$. In \cite{DGV2017} neither of the two players knows the true value of $\mcalJ$, so we have $(\ef^1_t)=(\ef^2_t)=(\ef^X_t)$, where $(\ef^X_t)$ is generated by the sample paths of the process $X$ and it is augmented by the $\prob$-null sets (notice that $\ef^X_t\subsetneq \ef_t$). In \cite{DEG2020} instead, the first player (minimiser) observes the true value of $\mcalJ$. In that case $(\ef^1_t)=(\ef_t)$ and $(\ef^2_t)=(\ef^X_t)$, so that $\ef^2_t\subsetneq \ef^1_t$. Using the notation $X^\mcalJ$ to emphasise the dependence of the underlying dynamics on $\mcalJ$, and given a choice of random times $\tau\in \te^R(\ef^1_t)$ and $\sigma \in \te^R(\ef^2_t)$ for the first and the second player, the game's payoff reads
\begin{equation*}
\mcal{P} (\tau, \sigma) = f(\tau,X^\mcalJ_\tau) \ind{\{\tau<\sigma\}} 
+
g (\sigma,X^\mcalJ_\sigma) \ind{\{{\sigma}<{\tau}\}}
+
h (\tau, X^\mcalJ_\tau) \ind{\{\tau = \sigma\}}.
\end{equation*}
Players assess the game by looking at the expected payoff as in \eqref{eq-uninf-payoff}. Finally, we remark that under a number of (restrictive) technical assumptions and with infinite horizon \cite{DGV2017} and \cite{DEG2020} show the existence of a value and of a saddle point in a smaller class of strategies. In \cite{DGV2017} both players use $(\ef^X_t)$-stopping times, with no need for additional randomisation. In \cite{DEG2020} the uninformed player uses $(\ef^X_t)$-stopping times but the informed player uses $(\ef_t)$-randomised stopping times.

\subsection{Game with two partially observed dynamics}
Here we show how the setting of \cite{GenGrun2019} also fits in our framework. This example is conceptually different from the previous two because the players observe two different stochastic processes. On a probability space $(\Omega,\ef,\prob)$ two processes $(X_t)$ and $(Y_t)$ are defined (in \cite{GenGrun2019} these are finite-state continuous-time Markov chains). The first player only observes the process $(X_t)$ while the second player only observes the process $(Y_t)$. In the notation of Section \ref{sec:setting}, we have $(\ef^1_t)=(\ef^X_t)$, $(\ef^2_t)=(\ef^Y_t)$ and $(\ef_t)=(\ef^X_t\vee\ef^Y_t)$, where the filtration $(\ef^X_t)$ is generated by the sample paths of $(X_t)$ and $(\ef^Y_t)$ by those of $(Y_t)$ (both filtrations are augmented with $\prob$-null sets). The payoff processes are deterministic functions of the underlying dynamics, i.e., $f_t=f(t,X_t,Y_t)$, $g_t=g(t,X_t,Y_t)$ and $h_t=h(t, X_t,Y_t)$, and they satisfy conditions \ref{eq-integrability-cond}-\ref{eq-terminal-time-order-cond}. Given a choice of random times $\tau \in \te^R(\ef^1_t)$ and $\sigma\in \te^R(\ef^2_t)$ for the first and the second player, the game's payoff reads
\begin{equation*}
\mcal{P} (\tau, \sigma) = f(\tau,X_\tau,Y_\tau) \ind{\{\tau<\sigma\}} 
+
g (\sigma,X_\sigma,Y_\sigma) \ind{\{{\sigma}<{\tau}\}}
+
h (\tau, X_\tau,,Y_\tau) \ind{\{\tau = \sigma\}}.
\end{equation*}
Players assess the game by looking at the expected payoff as in \eqref{eq-uninf-payoff}. We remark that the proof of existence of the value in \cite{GenGrun2019} is based on variational inequalities and relies on the finiteness of the state spaces of both underlying processes, and therefore cannot be extended to our general non-Markovian framework.

\subsection{Game with a random horizon}
Here we consider a non-Markovian extension of the framework of \cite{lempa2013}, where the time horizon of the game is exponentially distributed and independent of the payoff processes. On a probability space $(\Omega,\ef,\prob)$ we have a filtration $(\mcalG_t)_{t\in[0,T]}$, augmented with $\prob$-null sets, and a positive random variable $\theta$ which is independent of $\mcalG_T$ and has a continuous distribution. Let $\Lambda_t:=\ind{\{t\ge \theta\}}$ and take $\ef_t=\mcalG_t\vee\sigma(\Lambda_s,\,0\le s\le t)$.

The players have asymmetric knowledge of the random variable $\theta$. The first player observes the occurrence of $\theta$, whereas the second player does not. We have $(\ef^1_t)=(\ef_t)$ and $(\ef^2_t)=(\mcalG_t)\subsetneq (\ef^1_t)$.
Given a choice of random times $\tau \in \te^R(\ef^1_t)$ and $\sigma\in \te^R(\ef^2_t)$ for the first and the second player, the game's payoff reads
\begin{align}\label{eq:PLM}
\mcal{P} (\tau, \sigma) 
&= 
\indd{\tau \wedge \sigma \les \theta} \big(f^0_\tau \ind{\{\tau<\sigma\}} 
+
g^0_\sigma \ind{\{{\sigma}<{\tau}\}}
+
h^0_\tau \ind{\{\tau = \sigma\}} \big),
\end{align}
where $f^0$, $g^0$ and $h^0$ are $(\mcalG_t)$-adapted processes that satisfy conditions \ref{eq-integrability-cond}-\ref{eq-terminal-time-order-cond} and $f^0 \ge 0$.

Notice that the problem above does not fit directly into the framework of Section \ref{sec:setting}: Assumption \ref{eq-integrability-cond} is indeed violated, because the processes $(\indd{t \le \theta} f^0_t),(\indd{t \le \theta}g^0_t)$ are not \cadlag. However, we now show that the game can be equivalently formulated as a game satisfying conditions of our framework. The expected payoff can be rewritten as follows
\begin{align*}
N^0(\tau, \sigma) := \ee\big[\mcal{P} (\tau, \sigma) \big]
&= 
\ee\big[\indd{\tau \le \theta} \ind{\{\tau<\sigma\}} f^0_\tau 
+
\indd{\sigma \le \theta} \ind{\{{\sigma}<{\tau}\}} g^0_\sigma 
+
\indd{\sigma \le \theta} \ind{\{\tau = \sigma\}} h^0_\tau\big]\\
&= 
\ee\big[\indd{\tau \le \theta} \ind{\{\tau<\sigma\}} f^0_\tau 
+
\indd{\sigma < \theta} \ind{\{{\sigma}<{\tau}\}} g^0_\sigma 
+
\indd{\sigma < \theta} \ind{\{\tau = \sigma\}} h^0_\tau\big],\notag
\end{align*}
where the second equality holds because $\theta$ is continuously distributed and independent of $\ef^2_T$, so $\prob(\sigma=\theta) = 0$ for any $\sigma \in \te^R(\ef^2_t)$. Fix $\eps > 0$ and set
\begin{align*}
f^\eps_t:=f^0_{t}\ind{\{t<\theta+\eps\}},\quad g_t:=g^0_t\ind{\{t < \theta\}}, \quad h_t:=h^0_t\ind{\{t < \theta\}}, \qquad t \in [0, T].
\end{align*}
We see that conditions \ref{eq-integrability-cond}, \ref{eq-order-cond}, \ref{eq-terminal-time-order-cond} hold for the processes $(f^\eps_t)$, $(g_t)$, $(h_t)$ (for condition \ref{eq-order-cond} we use that $f^0 \ge 0$). Condition \ref{ass:regular} (regularity of payoffs $f^\eps$ and $g$) is satisfied, because $\theta$ has a continuous distribution, so it is a totally inaccessible stopping time for the filtration $(\ef_t)$ by \cite[Example VI.14.4]{RogersWilliams}. Therefore, by Theorem \ref{thm:main}, the game with expected payoff
\[
N^\eps(\tau, \sigma) =\ee\big[\mcal{P}^\eps(\tau,\sigma)\big]:= \ee \big[\ind{\{\tau<\sigma\}} f^\eps_\tau 
+
\ind{\{{\sigma}<{\tau}\}} g_\sigma 
+
\ind{\{\tau = \sigma\}} h_\tau\big]
\]
has a value and a pair of optimal strategies exists.

We now show that the game with expected payoff $N^0$ has the same value as the one with expected payoff $N^\eps$, for any $\eps > 0$. First observe that
\begin{align*}
N^\eps(\tau, \sigma) - N^0(\tau, \sigma) = \ee\big[\indd{\tau < \sigma} \indd{\theta < \tau < \theta + \eps} f^0_\tau\big] \ge 0
\end{align*}
by the assumption that $f^0 \ge 0$. Hence, 
\begin{equation}\label{eqn:N_eps_upper}
\inf_{\tau \in \te^R(\ef^1_t)} \sup_{\sigma \in \te^R(\ef^2_t)} N^\eps (\tau, \sigma)
\ge
\inf_{\tau \in \te^R(\ef^1_t)} \sup_{\sigma \in \te^R(\ef^2_t)} N^0 (\tau, \sigma).
\end{equation}
To derive an opposite inequality for the lower values, fix $\sigma \in \te^R(\ef^2_t)$. For $\tau \in \te^R(\ef^1_t)$, define
\[
\hat \tau =
\begin{cases}
\tau, & \tau \le \theta,\\
T, & \tau > \theta.
\end{cases}
\]
Then, using that $\mcal{P}^\eps(\tau,\sigma)=\mcal{P}(\tau,\sigma)$ on $\{\tau\le \theta\}$ and $\mcal{P}^\eps(T,\sigma)=g^0_\sigma\ind{\{\sigma<\theta\}}=\mcal{P}(\tau,\sigma)$ on $\{\tau>\theta\}$, we have
$N^\eps(\hat \tau, \sigma) = N^0(\tau, \sigma)$. It then follows that
\[
\inf_{\tau \in \te^R(\ef^1_t)} N^\eps (\tau, \sigma) \le \inf_{\tau \in \te^R(\ef^1_t)} N^0 (\tau, \sigma),
\]
which implies
\begin{equation}\label{eqn:N_eps_lower}
\sup_{\sigma \in \te^R(\ef^2_t)} \inf_{\tau \in \te^R(\ef^1_t)} N^\eps (\tau, \sigma) \le 
\sup_{\sigma \in \te^R(\ef^2_t)} \inf_{\tau \in \te^R(\ef^1_t)} N^0 (\tau, \sigma).
\end{equation}

Since the value of the game with expected payoff $N^\eps$ exists, combining \eqref{eqn:N_eps_upper} and \eqref{eqn:N_eps_lower} we see that the value of the game with expected payoff $N^0$ also exists. It should be noted, though, that this does not imply that an optimal pair of strategies for $N^\eps$ is optimal for $N^0$. 

It is worth noticing that in \cite{lempa2013} the setting is Markovian with $T=\infty$, $f^0_t=h^0_t=e^{-rt} \bar f(X_t)$, $g^0_t=e^{-rt} \bar g(X_t)$, $\bar f$, $\bar g$ deterministic functions, $r\ge 0$, $\theta$ exponentially distributed and $(X_t)$ a one-dimensional linear diffusion. Under specific technical requirements on the functions $\bar f$ and $\bar g$ the authors find that a pair of optimal strategies for the game \eqref{eq:PLM} exists when the first player uses $(\ef^1_t)$-stopping times and the second player uses $(\ef^2_t)$-stopping times (in the form of hitting times to thresholds), with no need for randomisation. Their methods rely on the theory of one-dimensional linear diffusions (using scale function and speed measure) and free-boundary problems, hence do not admit an extension to a non-Markovian case.

\section{Reformulation as a game of (singular) controls} \label{sec:reform}

In order to integrate out the randomisation devices for $\tau$ and $\sigma$ and obtain a reformulation of the payoff functional $N(\tau, \sigma)$ in terms of generating processes for randomised stopping times $\tau$ and $\sigma$, we need the following two auxiliary lemmata. We remark that if $\eta$ is a $(\mcalG_t)$-randomised stopping time for $(\mcalG_t) \subseteq (\ef_t)$, then $\eta$ is also an $(\ef_t)$-randomised stopping time. Therefore, the results below are formulated for $(\ef_t)$-randomises stopping times.

\begin{Lemma}\label{lem-eta-xi}
Let $\eta\in\te^R(\ef_t)$ with the generating process $(\rho_t)$. Then, for any $\ef_T$-measurable random variable $\kappa$ with values in $[0,T]$, 
\begin{alignat}{3}
&\ee[\ind{\{\eta\les \kappa\}}|\ef_T]=\rho_\kappa, \qquad &&\ee[\ind{\{\eta>\kappa\}}|\ef_T]=1-\rho_\kappa, \label{eq-xi-eta-1}\\  
&\ee[\ind{\{\eta<\kappa\}}|\ef_T]=\rho_{\kappa_-},\qquad &&\ee[\ind{\{\eta\ges \kappa\}}|\ef_T]=1-\rho_{\kappa_-}. \label{eq-xi-eta-3} 
\end{alignat}
\end{Lemma}
\begin{proof}
The proof of \eqref{eq-xi-eta-1} follows the lines of \cite[Proposition 3.1]{DEG2020}. Let $Z$ be the randomisation device for $\eta$. Since $\rho$ is right-continuous, non-decreasing and (\ref{eq-def-rand-st}) holds, we have
\begin{equation*}
\{\rho_\kappa > Z\}\subseteq \{\eta\les \kappa\}\subseteq\{\rho_\kappa\ges Z\}.
\end{equation*}
Using that $\rho_\kappa$ is $\ef_T$-measurable, and $Z$ is uniformly distributed and independent of $\ef_T$, we compute
\begin{equation*}
\ee[\ind{\{\eta\les \kappa\}}|\ef_T]\ges \ee[\ind{\{\rho_\kappa> Z\}}|\ef_T] = \int_0^1 \ind{\{\rho_\kappa> y\}} dy = \rho_\kappa,
\end{equation*}
and
\begin{equation*}
\ee[\ind{\{\eta\les \kappa\}}|\ef_T]\les \ee[\ind{\{\rho_\kappa\ges Z\}}|\ef_T] = \int_0^1 \ind{\{\rho_\kappa\ges y\}} dy = \rho_\kappa.
\end{equation*}
This completes the proof of the first equality in \eqref{eq-xi-eta-1}. The other one is a direct consequence.

To prove $(\ref{eq-xi-eta-3})$, we observe that, by (\ref{eq-xi-eta-1}), for any $\eps>0$ we have
\[
\ind{\{\kappa>0\}}\ee[\ind{\{\eta\les (\kappa-\eps) \vee (\kappa/2)\}}|\ef_T]=\ind{\{\kappa>0\}} \rho_{(\kappa-\eps) \vee (\kappa/2)}.
\]
Dominated convergence theorem implies 
\begin{align*}
\ee[\ind{\{\eta< \kappa\}}|\ef_T] &= \ind{\{\kappa > 0\}}\, \ee[\ind{\{\eta< \kappa\}}|\ef_T] 
= \lim_{\eps\downarrow 0} \ind{\{\kappa>0\}}\,\ee[\ind{\{\eta\les (\kappa-\eps) \vee (\kappa/2)\}}|\ef_T]\\
&= \lim_{\eps\downarrow 0} \ind{\{\kappa>0\}}\, \rho_{(\kappa-\eps) \vee (\kappa/2)} = \ind{\{\kappa>0\}}\, \rho_{\kappa-} = \rho_{\kappa-},
\end{align*}
where in the last equality we used that $\rho_{0-}=0$. This proves the first equality in \eqref{eq-xi-eta-3}. The other one is a direct consequence.
\end{proof}

\begin{Lemma}\label{lem:integ_out}
Let $\eta,\theta\in\te^R(\ef_t)$ with generating processes $(\rho_t)$, $(\chi_t)$ and independent randomisation devices $Z_\eta$, $Z_\theta$. For $(X_t)$ measurable, adapted and such that $\|X\|_{\mcalL}<\infty$ (but not necessarily {\cadlag\!\!}),
we have
\begin{align*}
&\ee\left[X_\eta \ind{\{\eta\les\theta\}\cap\{\eta<T\}}\right]=\ee\left[\int_{[0, T)} X_t(1-\chi_{t-})d\rho_t\right],\\
&\ee\left[X_\eta \ind{\{\eta<\theta\}}\right]=\ee\left[\int_{[0, T)} X_t(1-\chi_{t})d\rho_t\right],
\end{align*}
where we use the notation $\int_{[0, T)}$ for the (pathwise) Lebesgue-Stieltjes integral.
\label{lem-alt-repr-both-rand}
\end{Lemma}
\begin{proof}

For $y\in[0,1)$, define a family of random variables
\begin{equation*}
q(y)=\inf\{t\in [0, T]: \rho_t > y\}.
\end{equation*}
Then, $\eta=q(Z_\eta)$. Using that $Z_\eta \sim U(0,1)$ and Fubini's theorem, we see that
\begin{align*}
\ee\left[X_\eta \ind{\{\eta\les\theta\}\cap\{\eta<T\}}\right]&=\ee\left[\int_0^1 X_{q(y)} \ind{\{q(y)\les\theta\}\cap\{q(y)<T\}} dy\right]\\
&=\int_0^1 \ee\left[\ee\left[X_{q(y)} \ind{\{q(y)\les\theta\}\cap\{q(y)<T\}}|\ef_T\right]\right]dy.
\end{align*}
Since $X_{q(y)} \ind{\{q(y)<T\}}$ is $\ef_T$-measurable and the randomization device $Z_\theta$ is independent of $\ef_T$, we continue as follows:
\begin{align*}
\int_0^1 \ee\left[\ee\left[X_{q(y)} \ind{\{q(y)\les\theta\}\cap\{q(y)<T\}}|\ef_T\right]\right]dy&=\int_0^1 \ee\left[X_{q(y)} \ind{\{q(y)<T\}} \ee[\ind{\{q(y)\les\theta\}} | \ef_T]\right] dy\\
&=\ee\left[\int_0^1 X_{q(y)} \ind{\{q(y)<T\}} (1-\chi_{q(y)-}) dy\right]\\
&=\ee\left[\int_{[0, T)} X_t(1-\chi_{t-})d\rho_t\right],
\end{align*}
where in the second equality we apply Lemma \ref{lem-eta-xi} with $\kappa = q(y)$ and in the third equality we change the variable of integration applying \cite[Proposition 0.4.9]{revuzyor} $\omega$-wise and using the fact that the function $y \mapsto q(y)(\omega)$ is the generalized inverse of $t\mapsto \rho_t(\omega)$. The first statement of the lemma is now proved.

For the second statement, we adapt the arguments above to write
\begin{align*}
\ee\left[X_\eta \ind{\{\eta<\theta\}}\right] &= \int_0^1 \ee\left[X_{q(y)} \ee[\ind{\{q(y)<\theta\}} | \ef_T]\right] dy
=\ee\left[\int_0^1 X_{q(y)} (1-\chi_{q(y)}) dy\right]\\
&=\ee\left[\int_{[0, T]} X_t(1-\chi_{t})d\rho_t\right]
=\ee\left[\int_{[0, T)} X_t(1-\chi_{t})d\rho_t\right],
\end{align*}
where in the last equality we used that $\chi_T = 1$. 
\end{proof}

\begin{cor}\label{cor:j}
Under the assumptions of Lemma \ref{lem:integ_out}, we have
\[
\ee[X_\eta\ind{\{\eta=\theta\}}]=\ee\bigg[\sum_{t\in[0,T]}X_t\Delta\rho_t\Delta\chi_t\bigg],
\]
where $\Delta \rho_t = \rho_t - \rho_{t-}$ and $\Delta \chi_t = \chi_t - \chi_{t-}$.
\end{cor}
\begin{proof}
From Lemma \ref{lem-alt-repr-both-rand} we have
\begin{align*}
\ee[X_\eta\ind{\{\eta=\theta\}\cap\{\eta<T\}}]=&\ee\big[X_{\eta}\big(\ind{\{\eta\le\theta\}\cap\{\eta<T\}}-\ind{\{\eta<\theta\}}\big)\big]\\
=&\ee\bigg[\int_{[0,T)}X_t\Delta\chi_t d\rho_t\bigg]=\ee\bigg[\sum_{t\in[0,T)} X_t\Delta\chi_t \Delta\rho_t\bigg],
\end{align*}
where the final equality is due to the fact that $t\mapsto\chi_t(\omega)$ has countably many jumps for each $\omega \in \Omega$ and the continuous part of the measure $d\rho_t(\omega)$ puts no mass there. Further, we notice that 
\begin{align*}
\ee\big[\ind{\{\eta=\theta=T\}}|\ef_T\big]
&=
\lim_{n\to\infty}\ee\big[\ind{\{\eta>T-1/n\}}\ind{\{\theta>T-1/n\}}|\ef_T\big]
=
\lim_{n\to\infty}\ee\big[\ind{\{\rho_{T-1/n}\le Z_\eta\}}\ind{\{\chi_{T-1/n}\le Z_\theta\}}|\ef_T\big]\\
&=
\lim_{n\to\infty}(1-\rho_{T-1/n})(1-\chi_{T-1/n})
=
\Delta\rho_T\Delta\chi_T,
\end{align*}
where the second equality is by 
\[
\{\rho_{T-1/n}<Z_\eta\}\subseteq\{\eta>T-\tfrac{1}{n}\}\subseteq\{\rho_{T-1/n}\le Z_\eta\},
\]
and analogous inclusions for $\{\theta\!>\!T\!-\!\frac{1}{n}\}$. The third equality uses that $\rho_{T-1/n}$ and $\chi_{T-1/n}$ are $\ef_T$-measurable, and $Z_\eta$, $Z_\theta$ are independent of $\ef_T$. The final equality follows since $\rho_T=\chi_T=1$. Combining the above gives
the desired result.
\end{proof}

Applying Lemma \ref{lem-alt-repr-both-rand} and Corollary \ref{cor:j} to \eqref{eqn:payoff} and \eqref{eq-uninf-payoff}, we obtain the following reformulation of the game.

\begin{Proposition}\label{prop-functionals-equal}
For $\tau\in \te^R (\ef^1_t)$, $\sigma\in \te^R(\ef^2_t)$, 
\begin{equation}
N(\tau,\sigma)= \ee\bigg[\int_{[0, T)} f_t(1-\zeta_{t})d\xi_t + \int_{[0, T)} g_t(1-\xi_t)d\zeta_t + \sum_{t \in [0, T]} h_t \Delta\xi_t \Delta\zeta_t\bigg],
\label{eq-functional-in-terms-of-controls}
\end{equation}
where $(\xi_t)$ and $(\zeta_t)$ are the generating processes for $\tau$ and $\sigma$, respectively.
\end{Proposition}

With a slight abuse of notation, we will denote the right-hand side of \eqref{eq-functional-in-terms-of-controls} by $N(\xi,\zeta)$. 

\begin{remark}\label{rem-Laraki-Solan}
In the Definition \ref{def-value-rand-strat} of the lower value, the infimum can always be replaced by infimum over \emph{pure} stopping times (cf. \cite{LarakiSolan2005}). Same holds for the supremum in the definition of the upper value.

Let us look at the upper value: take arbitrary $\tau\in \te^R(\ef^1_t)$, $\sigma\in \te^R(\ef^2_t)$, and define the family of stopping times
\begin{equation*}
q(y)=\inf\{t\in [0,T]: \zeta_t > y\}, \qquad y \in [0,1),
\end{equation*}
similarly to the proof of Lemma \ref{lem-alt-repr-both-rand} and with $(\zeta_t)$ the generating process of $\sigma$. Then,
\begin{equation*}
N(\tau,\sigma)=\int_0^1 N(\tau,q(y)) dy \les \sup_{y\in[0,1)} N(\tau,q(y))\les \sup_{\sigma\in\te(\ef^2_t)} N(\tau,\sigma),
\end{equation*}
where $\te(\ef^2_t)$ denotes the set of pure $(\ef^2_t)$-stopping times. Since $\te(\ef^2_t) \subset \te^R(\ef^2_t)$, we have
\begin{equation*}
\sup_{\sigma\in\te^R(\ef^2_t)}N(\tau,\sigma)= \sup_{\sigma\in\te(\ef^2_t)} N(\tau,\sigma),
\end{equation*}
and, consequently, the `{\em inner}' optimisation can be done over pure stopping times:
\begin{equation*}
\inf_{\tau\in\te^R(\ef^1_t)}\sup_{\sigma\in\te^R(\ef^2_t)} N(\tau,\sigma)= \inf_{\tau\in \te^R(\ef^1_t)}\sup_{\sigma\in\te(\ef^2_t)} N(\tau,\sigma).
\end{equation*}
By the same argument one can show that
\begin{equation*}
\sup_{\sigma\in\te^R(\ef^2_t)} \inf_{\tau\in\te^R(\ef^1_t)} N(\tau,\sigma)= \sup_{\sigma\in\te^R(\ef^2_t)} \inf_{\tau\in \te(\ef^1_t)} N(\tau,\sigma).
\end{equation*}
However, in general an analogue result for the `{\em outer}' optimisation does not hold, i.e.,
\begin{equation*}
\sup_{\sigma\in\te^R(\ef^2_t)} \inf_{\tau\in \te^R(\ef^1_t)} N(\tau,\sigma)\neq \sup_{\sigma\in\te(\ef^2_t)} \inf_{\tau\in \te^R(\ef^1_t)} N(\tau,\sigma)
\end{equation*}
as shown by an example in Section \ref{sec:Nikita-examples}.
\end{remark}

\section{Sion's theorem and existence of value}\label{sec-Sion-existence-of-value}\label{sec:sions}
The proofs of Theorems \ref{thm:main2} and \ref{thm:main}, i.e., that the game with payoff \eqref{eq-uninf-payoff} has a value in randomised strategies, utilises Sion's min-max theorem \cite{Sion1958} (see also \cite{Komiya1988} for a simple proof). The idea of relying on Sion's theorem comes from \cite{TouziVieille2002} where the authors study zero-sum Dynkin games with full and symmetric information. Here, however, we need different key technical arguments as explained in, e.g., Remark \ref{rem-TV-norm-doesnt-work} below.

Let us start by recalling Sion's theorem.
\begin{theorem}[Sion's theorem]\label{th-the-Sion}
\cite[Corollary 3.3]{Sion1958}
Let $A$ and $B$ be convex subsets of a linear topological space one of which is compact. Let $\varphi(\mu,\nu)$ be a function $A\times B \mapsto \er$ that is quasi-concave and upper semi-continuous in $\mu$ for each $\nu\in B$, and quasi-convex and lower semi-continuous in $\nu$ for each $\mu\in A$. Then,
\begin{equation*}
\sup_{\mu\in A}\inf_{\nu\in B} \varphi(\mu,\nu)=\inf_{\nu\in B}\sup_{\mu\in A} \varphi(\mu,\nu).
\end{equation*}
\end{theorem}

The key step in applying Sion's theorem is to find a topology on the set of randomised stopping times, or, equivalently, on the set of corresponding generating processes so that the functional $N(\cdot, \cdot)$ satisfies the assumptions. We will use the weak topology of 
\[
\es := L^2 \big([0, T] \times \Omega, \mathcal{B}([0, T]) \times \ef, \lambda \times \prob\big),
\]
where $\lambda$ denotes the Lebesgue measure on $[0, T]$. Given a filtration $(\mcalG_t) \subseteq (\ef_t)$, in addition to the class of increasing processes $\mcalAcirc(\mcalG_t)$ introduced in Section \ref{sec:setting}, here we also need
\begin{align*}
\mcalAcircAc (\mcalG_t) :=&\,\{\rho\in \mcalAcirc(\mcalG):\,\text{$t\mapsto\rho_t(\omega)$ is absolutely continuous on $[0,T)$ for all $\omega\in\Omega$}\}.
\end{align*}
It is important to notice that $\rho\in\mcalAcircAc(\mcalG_t)$ may have a jump at time $T$ if
\[
\rho_{T-}(\omega):=\lim_{t\uparrow T}\int_0^t\big(\tfrac{\ud}{\ud t}\rho_s\big)(\omega)\ud s<1=\rho_T(\omega).
\] 
As with $\mcalAcirc(\mcalG_t)$, in the definition of $\mcalAcircAc(\mcalG_t)$ we require that the stated properties hold for all $\omega \in \Omega$, which causes no loss of generality if $\mcalG_0$ contains all $\prob$-null sets of $\Omega$. It is clear that $\mcalAcircAc (\mcalG_t) \subset\mcalAcirc(\mcalG_t) \subset\mathcal{S}$.

For reasons that will become clear later (e.g., see Lemma \ref{lem-strat-set-compact}), we prefer to work with slightly more general processes than those in $\mcalAcirc(\mcalG_t)$ and $\mcalAcircAc(\mcalG_t)$. Let us denote
\begin{align*}
\laa(\mcalG_t) :=&\, \{ \rho \in \mathcal S : \,\exists\; \hat\rho\in\mcalAcirc (\mcalG_t) \,\text{such that $\rho = \hat \rho$ for $(\lambda \times \prob)$\ae $(t,\omega)\in[0,T]\times\Omega$}\},\\
\laac(\mcalG_t) := &\,\{ \rho \in \mathcal S : \,\exists\; \hat\rho\in\mcalAcircAc (\mcalG_t) \,\text{such that $\rho = \hat \rho$ for $(\lambda \times \prob)$\ae $(t,\omega)\in[0,T]\times\Omega$}\}.
\end{align*}
We will call $\hat \rho$ in the definition of the set $\laa$ (and $\laac$) the \emph{\cadlag} (and {\em absolutely continuous}) {\em representative} of $\rho$. Although it is not unique, all \cadlag representatives are indistinguishable (Lemma \ref{lem:cadlag_indis}). Hence, all \cadlag representatives $\hat\rho$ of $\rho\in\laa$ define the same positive measure on $[0,T]$ for $\prob$\ae $\omega\in\Omega$ via a non-decreasing mapping $t\mapsto\hat\rho_t(\omega)$.
Then, given any bounded measurable process $(X_t)$ the stochastic process (Lebesgue-Stieltjes integral)
\begin{equation*}
t \mapsto \int_{[0, t]} X_s\, \ud\hat \rho_s, \qquad t \in [0, T],  
\end{equation*}
does not depend on the choice of the \cadlag representative $\hat \rho$ in the sense that it is defined up to indistinguishability.

The next definition connects the randomised stopping times that we use in the construction of the game's payoff (Proposition \ref{prop-functionals-equal}) with processes from the classes $\laa(\ef^1_t)$ and $\laa(\ef^2_t)$. Note that $\laa(\mcalG_t) \subseteq \laa(\ef_t)$ whenever $(\mcalG_t) \subseteq (\ef_t)$, so the definition can be stated for $\laa(\ef_t)$ without any loss of generality.
\begin{definition}\label{def:integral}
Let $(X_t)$ be measurable and such that $\|X\|_{\mcalL}\!<\!\infty$ (not necessarily {\cadlag\!\!}). For $\chi,\rho \in \laa(\ef_t)$, we define the Lebesgue-Stieltjes integral processes
\[
t \mapsto \int_{[0, t]} X_s\, \ud\rho_s,\quad t\mapsto\int_{[0, t]} X_s\,(1-\chi_{s}) \ud\rho_s\quad\text{and}\quad t\mapsto\int_{[0, t]} X_s\,(1-\chi_{s-}) \ud\rho_s \qquad t \in [0, T],   
\]
by 
\[
t \mapsto \int_{[0, t]} X_s\, \ud\hat{\rho}_s,\quad t\mapsto\int_{[0, t]} X_s\,(1-\hat{\chi}_{s}) \ud\hat{\rho}_s\quad\text{and}\quad t\mapsto\int_{[0, t]} X_s\,(1-\hat{\chi}_{s-}) \ud\hat{\rho}_s \qquad t \in [0, T],   
\]
for any choice of the \cadlag representatives $\hat \rho$ and $\hat \chi$, uniquely up to indistinguishability.
\end{definition}

With a slight abuse of notation we define a functional $N: \laa(\ef^1_t) \times \laa(\ef^2_t) \to \er$ by the right-hand side of \eqref{eq-functional-in-terms-of-controls}. It is immediate to verify using Definition \ref{def-value-rand-strat} and Proposition \ref{prop-functionals-equal} that the lower and the upper value of our game satisfy
\begin{align}\label{eq:VV}
V_{*}=\sup_{\zeta\in\laa(\ef^2_t)}\inf_{\xi\in\laa(\ef^1_t)} N(\xi,\zeta), \qquad V^*=\inf_{\xi\in\laa(\ef^1_t)} \sup_{\zeta\in\laa(\ef^2_t)} N(\xi,\zeta).
\end{align}
Notice that even though according to Definition \ref{def-rand-st} the couple $(\xi,\zeta)$ should be taken in $\mcalAcirc(\ef^1_t)\times\mcalAcirc(\ef^2_t)$, in \eqref{eq:VV} we consider $(\xi,\zeta)\in \laa(\ef^1_t)\times\laa(\ef^2_t)$. This causes no inconsistency thanks to the discussion above and Definition \ref{def:integral} for integrals.

\begin{remark} 
The mapping $\laa(\ef^1_t) \times \laa(\ef^2_t) \ni (\xi, \zeta) \mapsto N(\xi, \zeta)$ does not satisfy the conditions of Sion's theorem under the strong or the weak topology of $\es$. Indeed, taking $\xi^n_t = \ind{\{t \ge T/2 + 1/n\}}$, we have $\xi^n_t \to \ind{\{t \ge T/2\}}=:\xi_t$ for $\lambda$\ae $t \in [0, T]$, so that by the dominated convergence theorem $(\xi^n)$ also converges to $\xi$ in $\es$. Then, fixing $\zeta_t = \ind{\{t \ge T/2\}}$ in $\laa(\ef^2_t)$ we have $N(\xi^n, \zeta) = \ee[g_{T/2}]$ for all $n\ge 1$ whereas $N(\xi, \zeta) =\ee[ h_{T/2} ]$. So the lower semicontinuity of $\xi \mapsto N(\xi, \zeta)$ cannot be ensured if, for example, $\prob(h_{T/2}>g_{T/2})>0$.
\end{remark}

Due to issues indicated in the above remark, as in \cite{TouziVieille2002}, we `smoothen' the control strategy of one player in order to introduce additional regularity in the payoff. We will show that this procedure does not change the value of the game (Proposition \ref{thm:conv_lipsch}). We choose (arbitrarily and with no loss of generality, thanks to Remark \ref{rem:ineq}) to consider an auxiliary game in which the  first player can only use controls from $\laac (\ef^1_t)$. Let us define the associated upper/lower values:
\begin{equation}\label{eq-value-cont-restriction}
W_{*}=\sup_{\zeta\in\laa(\ef^2_t)}\inf_{\xi\in\laac(\ef^1_t)} N(\xi,\zeta)\quad\text{and}\quad W^*=\inf_{\xi\in\laac(\ef^1_t)} \sup_{\zeta\in\laa(\ef^2_t)} N(\xi,\zeta).
\end{equation}

Here, we work under the regularity assumption on the payoff processes \ref{ass:regular}. Relaxation of this assumption is conducted in Section \ref{sec:relax}.
The main results can be distilled into the following theorems:

\begin{theorem}\label{th-value-cont-strat}
Under assumptions \ref{eq-integrability-cond}-\ref{ass:filtration}, the game (\ref{eq-value-cont-restriction}) has a value, i.e.
\begin{equation*}
W_{*}=W^{*}:=W.
\end{equation*}
Moreover, the $\zeta$-player (maximiser) has an optimal strategy, i.e. there exists $\zeta^*\in\laa(\ef^2_t)$ such that
\begin{equation*}
\inf_{\xi\in\laac(\ef^1_t)} N(\xi,\zeta^*)=W.
\end{equation*}
\end{theorem}
\begin{Proposition}\label{thm:conv_lipsch}
Under assumptions \ref{eq-integrability-cond}-\ref{ass:filtration}, for any $\zeta \in \laa(\ef^2_t)$ and $\xi \in \laa(\ef^1_t)$, there is a sequence $\xi^n \in \laac(\ef^1_t)$ such that
\[
\limsup_{n \to \infty} N(\xi^n, \zeta) \le N(\xi, \zeta).
\]
\end{Proposition}

The proofs of the above theorems will be conducted in the following subsections: Section \ref{sec:tech} contains a series of technical results which we then use to prove Theorem \ref{th-value-cont-strat} (in Section \ref{sec:verif}) and Proposition \ref{thm:conv_lipsch} (in Section \ref{sec:approx}). With the results from Theorem \ref{th-value-cont-strat} and Proposition \ref{thm:conv_lipsch} in place we can provide a (simple) proof of Theorem \ref{thm:main}.
\begin{proof}[{\bf Proof of Theorem \ref{thm:main}}]
Obviously, $V_* \le W_*$ and $V^* \le W^*$. However, Proposition \ref{thm:conv_lipsch} implies that 
\begin{align}\label{eq:W*}
\inf_{\xi\in\laac(\ef^1_t)} N(\xi,\zeta) = \inf_{\xi\in\laa(\ef^1_t)} N(\xi,\zeta)\quad\text{for any $\zeta\in\laa(\ef^2_t)$},
\end{align} 
so $V_* \ge W_*$ and therefore $V_* = W_*$. Then, thanks to Theorem \ref{th-value-cont-strat}, we have a sequence of inequalities which completes the proof of existence of the value
\[
W = W_* = V_* \le V^* \le W^* = W.
\]
In \eqref{eq:W*} we can choose $\zeta^*$ which is optimal for $W$ (its existence is guaranteed by Theorem \ref{th-value-cont-strat}). Then,
\[
V=V_*=\inf_{\xi\in\laa(\ef^1_t)} N(\xi,\zeta^*).
\]
Thanks to Remark \ref{rem:ineq}, we can repeat the same arguments above with the roles of the two players swapped as in \eqref{eq:swap}, i.e., the $\tau$-player ($\xi$-player) is the maximiser and the $\sigma$-player ($\zeta$-player) is the minimiser. Thus, applying again Theorem \ref{th-value-cont-strat} and Proposition \ref{thm:conv_lipsch} (with $\mcal{P}'$ as in Remark \ref{rem:ineq} in place of $\mcal{P}$) we arrive at 
\[
-V=:V'=\inf_{\zeta\in\laa(\ef^2_t)} \ee\big[\mcal{P}'(\xi^*,\zeta)\big],
\]
where $\xi^*\in\laa(\ef^1_t)$ is optimal for the maximiser in the game with value $W'=-W$. Hence $\xi^*$ is optimal for the minimiser in the original game with value $V$ and the couple $(\xi^*,\zeta^*)\in\laa(\ef^1_t)\times\laa(\ef^2_t)$ is a saddle point. The corresponding randomised stopping times, denoted $(\tau_*,\sigma_*)$, are an optimal pair for the players. 
\end{proof}


\subsection{Technical results}\label{sec:tech}

In this section we give a series of results concerning the convergence of integrals when either the integrand or the integrator converge in a suitable sense. We start by stating a technical lemma whose easy proof is omitted. 
\begin{Lemma}\label{lem:cadlag_indis}
Let $(X_t)$ and $(Y_t)$ be \cadlag measurable processes such that $X_t = Y_t$, $\prob$\as for $t \in D \subset [0, T)$ countable and dense, $X_{0-} = Y_{0-}$ and $X_T = Y_T$, $\prob$\as Then $(X_t)$ is indistinguishable from $(Y_t)$.
\end{Lemma}

\begin{definition}\label{def:def_C}
Given a \cadlag measurable process $(X_t)$, for each $\omega\in\Omega$ we denote
\[
C_X(\omega):= \{ t\in[0,T]: X_{t-}(\omega)=X_t(\omega) \}.
\]
\end{definition}
Our next result tells us that the convergence $(\lambda\times\mathbb{P})$\ae of processes in $\laa (\mcalG_t)$ can be lifted to $\prob$\as convergence at all points of continuity of the corresponding \cadlag representatives.
\begin{Lemma}\label{lem:cadlag_convergence}
For a filtration $(\mcalG_t) \subseteq (\ef_t)$, let $(\rho^n)_{n\ge 1}\subset\laa(\mcalG_t)$ and $\rho \in \laa(\mcalG_t)$ with $\rho^n \to \rho$ $(\lambda \times \prob)$\ae as $n\to\infty$. 
Then for any \cadlag representatives $\hat \rho^n$ and $\hat \rho$ we have
\begin{equation}\label{eqn:cadlag_convergence}
\prob\Big(\big\{\omega \in \Omega:\ \lim_{n\to\infty}\hat \rho^n_t(\omega)= \hat \rho_t(\omega) \:\:\text{for all $t\in C_{\hat \rho}(\omega)$}\big\}\Big) = 1.
\end{equation}
\end{Lemma}
\begin{proof}
The $(\lambda \times \prob)$\ae convergence of $\rho^n$ to $\rho$ means that the \cadlag representatives $\hat \rho_n$ converge to $\hat \rho$ also $(\lambda \times \prob)$\ae. Hence, there is a set $D \subset [0, T]$ with $\lambda([0,T]\setminus D) = 0$ such that $\hat \rho^n_t \to \hat \rho_t$ $\prob$\as for $t \in D$. Since $\lambda([0,T]\setminus D) = 0$, there is a countable subset $D_0 \subset D$ that is dense in $[0, T]$. Define 
\[
\Omega_0 := \{ \omega \in \Omega:\ \hat \rho^n_t (\omega) \to \hat \rho_t (\omega)\:\: \text{for all $t \in D_0$}\}. 
\]
Then $\prob(\Omega_0) = 1$.

Now, fix $\omega\in\Omega_0$ and let $t\in C_{\hat \rho}(\omega) \cap (0, T)$. Take an increasing sequence $(t^1_k)_{k\ge 1}\subset D_0$ and a decreasing one $(t^2_k)_{k\ge 1}\subset D_0$, both converging to $t$ as $k\to\infty$. For each $k\ge 1$ we have
\begin{equation}\label{eqn:upward_conv}
\hat \rho_t(\omega)=\lim_{k\to\infty}\hat \rho_{t^2_k}(\omega)=\lim_{k\to\infty}\lim_{n\to\infty}\hat \rho^n_{t^2_k}(\omega)\ge \limsup_{n\to\infty}\hat \rho^n_t(\omega), 
\end{equation}
where in the final inequality we use that $\hat \rho^n_{t^2_k}(\omega)\ge \hat\rho^n_t(\omega)$ by monotonicity. By analogous arguments we also obtain
\[
\hat \rho_t(\omega)=\lim_{k\to\infty}\hat \rho_{t^1_k}(\omega)=\lim_{k\to\infty}\lim_{n\to\infty}\hat \rho^n_{t^1_k}(\omega)\le \liminf_{n\to\infty}\hat \rho^n_t(\omega), 
\]
where the first equality holds because $t\in C_{\hat \rho}(\omega)$. Combining the above we get \eqref{eqn:cadlag_convergence} (apart from $t\in\{0,T\}$) by recalling that $\omega\in\Omega_0$ and $\prob(\Omega_0)=1$. The convergence at $t = T$, irrespective of whether it belongs to $C_{\hat\rho}(\omega)$, is trivial as $\hat \rho^n_T(\omega) = \hat \rho_T(\omega) = 1$. If $0 \in C_{\hat\rho}(\omega)$, then $\hat \rho_0 (\omega) = \hat \rho_{0-}(\omega) = 0$. Inequality \eqref{eqn:upward_conv} reads $0 = \hat \rho_0 (\omega) \ge \limsup_{n \to \infty} \hat \rho^n_0(\omega)$. Since $\hat \rho^n_0(\omega) \ge 0$, this proves that $\hat\rho^n_0(\omega) \to \hat \rho_0(\omega)=0$.
\end{proof}

\begin{Lemma}\label{prop-terminal-time-jump-limit}
For a filtration $(\mcalG_t) \subseteq (\ef_t)$, let $(\rho^n)_{n\ge 1}\subset\mcalAcirc(\mcalG_t)$ and $\rho\in\mcalAcirc(\mcalG_t)$ with $\rho^n\to \rho$ $(\lambda\times\prob)$\ae as $n\to\infty$. For any $t \in [0, T]$ and any random variable $X \ge 0$ with $\ee[X]<\infty$, we have
\begin{equation*}
\limsup_{n\to\infty} \ee[X\Delta \rho^n_t]\les \ee[X\Delta \rho_t].
\end{equation*}
\end{Lemma}
\begin{proof}
Fix $t \in (0, T)$. Using $(\lambda\times\prob)$\ae convergence of $\rho^{n}$ to $\rho$, i.e., that $\int_0^T \prob\big(\lim_{n\to\infty}\rho^{n}_t=\rho_t\big) \ud t=T$, there is a decreasing sequence $\delta_m \to 0$ such that
\begin{equation*}
\lim_{n\to\infty} \rho^{n}_{t-\delta_m} = \rho_{t-\delta_m},
\qquad
\lim_{n\to\infty} \rho^{n}_{t+\delta_m} = \rho_{t+\delta_m},\qquad \prob\as
\end{equation*}
Then, by the dominated convergence theorem,
\begin{align*}
\ee[X\Delta \rho_t]&=\lim_{m \to \infty} \ee[ X (\rho_{t + \delta_m} - \rho_{t - \delta_m})]\\
&=
\lim_{m\to\infty}\lim_{n\to\infty} \ee[X(\rho^{n}_{t+\delta_m}-\rho^{n}_{t-\delta_m})]\\
&=
\lim_{m\to\infty}\limsup_{n\to\infty} \ee[X(\rho^{n}_{t+\delta_m}-\rho^{n}_{t-\delta_m})]\\
&=
\lim_{m\to\infty}\limsup_{n\to\infty} \ee[X(\rho^{n}_{t+\delta_m}-\rho^{n}_{t} + \rho^{n}_{t-} - \rho^{n}_{t-\delta_m} + \Delta \rho^{n}_{t})] \ges \limsup_{n\to\infty} \ee[X \Delta \rho^{n}_t],
\end{align*}
where the last inequality is due to $t\mapsto\rho^{n}_t$ being non-decreasing. This finishes the proof for $t\in(0,T)$.
The proof for $t \in\{ 0, T\}$ is a simplified version of the argument above, since $\rho^n_T=\rho_T=1$ and $\rho^n_{0-}=\rho_{0-}=0$, $\prob$\as
\end{proof}

We need to consider a slightly larger class of processes $\tlmcalAcirc(\mcalG_t) \supset \mcalAcirc(\mcalG_t)$ defined by
\begin{align*}
\tlmcalAcirc(\mcalG_t):=&\,\{\rho\,:\,\text{$\rho$ is $(\mcalG_t)$-adapted with $t\mapsto\rho_t(\omega)$ \cadlag,}\\
&\qquad\,\text{non-decreasing, $\rho_{0-}(\omega)=0$ and $\rho_T(\omega)\le 1$ for all $\omega\in\Omega$}\}.
\end{align*}

\begin{Proposition}\label{prop:r-convergence}
For a filtration $(\mcalG_t) \subseteq (\ef_t)$, let $(\rho^n)_{n\ge 1}\subset\tlmcalAcirc(\mcalG_t)$ and $\rho\in\tlmcalAcirc(\mcalG_t)$. Assume
\[
\prob\Big(\big\{\omega \in \Omega:\ \lim_{n\to\infty}\rho^n_t(\omega)=\rho_t(\omega)\:\:\text{for all $t\in C_\rho(\omega)\cup\{T\}$} \big\} \Big) = 1.
\]
Then for any $X\in\mcalL$ that is also $(\mathcal F_t)$-adapted and regular, we have
\begin{equation}\label{eqn:b_t}
\lim_{n\to\infty}\ee\bigg[\int_{[0, T]} X_t \ud\rho^n_t\bigg] = \ee\bigg[\int_{[0, T]} X_t \ud\rho_t\bigg].
\end{equation}
\end{Proposition}
\begin{proof}
Let us first assume that $(X_t) \in \mcalL$ has continuous trajectories (but is not necessarily adapted). If we prove that 
\begin{equation}\label{eqn:b_t_omega}
\lim_{n\to\infty}\int_{[0, T]} X_t(\omega) \ud \rho^n_t(\omega)= \int_{[0, T]} X_t(\omega) \ud \rho_t (\omega),\quad\text{for $\mathbb{P}$\ae $\omega\in\Omega$,}
\end{equation}
then the result in \eqref{eqn:b_t} will follow by the dominated convergence theorem. By assumption there is $\Omega_0\subset \Omega$ with $\prob(\Omega_0)=1$ and such that $\rho^n_t(\omega)\to\rho_t(\omega)$ at all points of continuity of $t\mapsto \rho_t(\omega)$ and at the terminal time $T$ for all $\omega \in \Omega_0$. Since $\ud\rho^n_t(\omega)$ and $\ud\rho_t(\omega)$ define positive measures on $[0,T]$ for each $\omega\in\Omega_0$, the convergence of integrals in \eqref{eqn:b_t_omega} can be deduced from the weak convergence of finite measures, see \cite[Remark III.1.2]{Shiryaev}. Indeed, if $\omega\in\Omega_0$ is such that $\rho_T(\omega)=0$, the right-hand side of \eqref{eqn:b_t_omega} is zero and we have
\begin{equation*}
\limsup_{n\to\infty}\left|\int_{[0,T]} X_t(\omega) \ud \rho^n_t(\omega)\right| \le \limsup_{n\to\infty}\sup_{t \in [0, T]} |X_t(\omega)| \rho^n_T(\omega)= 0,
\end{equation*}
where we use $X\in\mcalL$ to ensure that $\sup_{t \in [0, T]} |X_t(\omega)|<\infty$. If instead, $\omega\in\Omega_0$ is such that $\rho_T(\omega)>0$, then for all sufficiently large $n$'s, we have $\rho^n_T(\omega) > 0$ and $t \mapsto \rho^n_t(\omega) / \rho^n_T(\omega)$ define cumulative distribution functions (cdfs) converging pointwise to $\rho_t(\omega) / \rho_T(\omega)$ at the points of continuity of $\rho_t(\omega)$. Since $t \mapsto X_t(\omega)$ is continuous, \cite[Thm III.1.1]{Shiryaev} justifies
\begin{align*}
\lim_{n\to\infty}\int_{[0, T]} X_t (\omega) \ud \rho^{n}_t(\omega) =&\,\lim_{n\to\infty} \rho^{n}_T (\omega) \int_{[0, T]} X_t(\omega) \ud\left(\frac{\rho^{n}_t(\omega)}{\rho^{n}_T(\omega)}\right) \\
=&\,\rho_T(\omega) \int_{[0, T]} X_t(\omega) \ud\left(\frac{\rho_t(\omega)}{\rho_T(\omega)}\right) = \int_{[0, T]} X_t(\omega) \ud \rho_t(\omega).
\end{align*}

Now we drop the continuity assumption on $X$. We turn our attention to \cadlag, $(\ef_t)$-adapted and regular $(X_t) \in \mcalL$. By \cite[Theorem 3]{Bismut1978} there is $(\tilde X_t) \in \mcalL$ with continuous trajectories (not necessarily adapted) such that $(X_t)$ is an $(\ef_t)$-optional projection of $(\tilde X_t)$. From the first part of the proof we know that \eqref{eqn:b_t} holds for $(\tilde X_t)$. To show that it holds for $(X_t)$ it is sufficient to notice that $(\rho^n_t)$ and $(\rho_t)$ are $(\ef_t)$-optional processes, and apply \cite[Thm VI.57]{DellacherieMeyer} to obtain
\[
 \ee\bigg[\int_{[0, T]} X_t \ud \rho^n_t\bigg] = \ee\bigg[\int_{[0, T]} \tilde X_t \ud \rho^n_t\bigg]\qquad \text{and} \qquad 
 \ee\bigg[\int_{[0, T]} X_t \ud \rho_t\bigg] = \ee\bigg[\int_{[0, T]} \tilde X_t \ud \rho_t\bigg].
\]
\end{proof}

\begin{remark}
The statement of Proposition \ref{prop:r-convergence} can be strengthened to include all processes in $\mcalL$ which are regular but not necessarily $(\ef_t)$-adapted. One can prove it by adapting arguments of the proof of \cite[Thm.\ 3]{Meyer}. 
\end{remark}

\begin{Proposition}\label{prop-specific-convergence-2}
For a filtration $(\mcalG_t) \subseteq (\ef_t)$, let $\chi\in \mcalAcirc(\mcalG_t)$ and $\rho\in\laac(\mcalG_t)$ and consider $X\in\mcalL$ which is $(\ef_t)$-adapted and regular. If $(\rho^n)_{n\ge 1}\subset\mathcal{A}_{ac}(\mcalG_t)$ converges $(\lambda \times \prob)$\ae to $\rho$ as $n\to\infty$, then
\begin{equation}\label{eq:lim00}
\lim_{n\to\infty}\ee\bigg[\int_{[0, T]} X_t(1-\chi_{t-})\ud\rho^n_t\bigg]=\ee\bigg[\int_{[0, T]} X_t(1-\chi_{t-})\ud\rho_t\bigg].
\end{equation}
\end{Proposition}
\begin{proof}
Define absolutely continuous adapted processes
\begin{equation*}
R^n_t = \int_{[0, t]} (1-\chi_{s-})d\rho^n_s\quad\text{and}\quad R_t = \int_{[0, t]} (1-\chi_{s-})d\rho_s,
\end{equation*}
so that
\begin{equation}\label{eq:intdR}
\int_{[0, T]} X_t(1-\chi_{t-})d\rho^n_t=\int_{[0, T]} X_tdR^n_t\quad\text{and}\quad \int_{[0, T]} X_t(1-\chi_{t-})d\rho_t=\int_{[0, T]} X_tdR_t.
\end{equation}
With no loss of generality we can consider the absolutely continuous representatives of $\rho$ and $\rho^n$ from the class $\mcalAcircAc(\mcalG_t)$ in the definition of all the integrals above (which we still denote by $\rho$ and $\rho^n$ for simplicity). In light of this observation it is clear that $(R^n)_{n\ge 1}\subset\tlmcalAcirc(\mcalG_t)$ and $R\in \tlmcalAcirc(\mcalG_t)$. The idea is then to apply Proposition \ref{prop:r-convergence} to the integrals with $R^n$ and $R$ in \eqref{eq:intdR}.

Thanks to Lemma \ref{lem:cadlag_convergence} and recalling that $\rho^n_T = \rho_T = 1$, the set
\[
\Omega_0 = \big\{\omega\in\Omega:\lim_{n\to\infty}\rho^n_t(\omega)= \rho_t(\omega) \text{ for all $t \in [0, T]$}\big\}
\]
has full measure, i.e., $\prob(\Omega_0)=1$. For any $\omega \in \Omega_0$ and $t\in[0,T]$, integrating by parts (see, e.g., \cite[Prop. 4.5, Chapter 0]{revuzyor}), using the dominated convergence theorem and then again integrating by parts give
\begin{equation}\label{eqn:conv_R}
\lim_{n \to \infty} R^n_t= \lim_{n \to \infty} \bigg[(1-\chi_{t})\rho^n_t - \int_{[0,t]} \rho^n_sd(1-\chi_{s}) \bigg]
= (1-\chi_{t})\rho_t - \int_{[0,t]} \rho_sd(1-\chi_{s})=R_t.
\end{equation}
Hence $R^n$ and $R$ satisfy the assumptions of Proposition \ref{prop:r-convergence} and we can conclude that \eqref{eq:lim00} holds.
\end{proof}

\begin{cor}\label{cor-specific-convergence-2}
Under the assumptions of Proposition \ref{prop-specific-convergence-2}, we have 
\begin{equation*}
\lim_{n\to\infty}\ee\bigg[\int_{[0, T)} X_t(1-\chi_{t})\ud\rho^n_t + X_T \Delta \chi_T \Delta \rho^n_T\bigg]=\ee\bigg[\int_{[0, T)} X_t(1-\chi_{t})\ud\rho_t + X_T \Delta \chi_T \Delta \rho_T\bigg].
\end{equation*}
\end{cor}
\begin{proof}
Recall that $\rho^n$ and $\rho$ are continuous everywhere apart from $T$. Hence, we can rewrite the left- and right-hand side of \eqref{eq:lim00} as
\[
\int_{[0, T]} X_t(1-\chi_{t-})\ud\rho^n_t = \int_{[0, T]} X_t(1-\chi_{t})\ud\rho^n_t + X_T \Delta \chi_T \Delta \rho^n_T
\]
and
\[
\int_{[0, T]} X_t(1-\chi_{t-})\ud\rho_t = \int_{[0, T]} X_t(1-\chi_{t})\ud\rho_t + X_T \Delta \chi_T \Delta \rho_T\,,
\]
respectively. It remains to note that $\int_{[0, T]} X_t(1-\chi_{t})\ud\rho^n_t = \int_{[0, T)} X_t(1-\chi_{t})\ud\rho^n_t$ because $\chi_T = 1$.
\end{proof}

We close this technical section with a similar result to the above but for approximations which are needed for the proof of Proposition \ref{thm:conv_lipsch}. The next proposition is tailored for our specific type of regularisation of processes in $\laa(\ef^1_t)$. Notice that the left hand side of \eqref{eq:lim-in-A} features $\chi_{t-}$ while the right hand side has $\chi_t$.

\begin{Proposition}\label{prop-specific-convergence-3}
For a filtration $(\mcalG_t) \subseteq (\ef_t)$, let $\chi,\rho \in\mcalAcirc (\mcalG_t)$, $(\rho^n)_{n\ge 1}\subset\mcalAcirc (\mcalG_t)$ and consider $X\in\mcalL$ which is $(\ef_t)$-adapted and regular. Assume the sequence $(\rho^n)_{n\ge 1}$ is non-decreasing and for $\prob$\ae $\omega\in\Omega$
\begin{align}\label{eq:conv-rho}
\lim_{n\to\infty}\rho^{n}_t(\omega)=\rho_{t-}(\omega)\:\: \text{for all $t\in[0,T)$}.
\end{align}
Then
\begin{equation}\label{eq:lim-in-A}
\lim_{n\to\infty}\ee\bigg[\int_{[0, T)} X_t(1-\chi_{t-})\ud\rho^n_t\bigg]= \ee\bigg[\int_{[0, T)} X_t(1-\chi_{t})\ud\rho_t\bigg],
\end{equation}
and for $\prob$\ae $\omega \in \Omega$
\begin{equation}\label{eqn:lim-in-t-}
\lim_{n\to\infty}\rho^n_{t-}(\omega)=\rho_{t-}(\omega)\quad \text{for all $t \in [0, T]$}.
\end{equation}
\end{Proposition}
\begin{proof}
Denote by $\Omega_0$ the set on which the convergence \eqref{eq:conv-rho} holds. The first observation is that for all $\omega\in\Omega_0$ and $t \in (0, T]$
\begin{align}\label{eq:lim-nt}
\lim_{n\to\infty}\rho^n_{t-}(\omega)=\lim_{n\to\infty}\lim_{u\uparrow t}\rho^n_{u}(\omega)=\lim_{u\uparrow t}\lim_{n\to\infty}\rho^n_{u}(\omega)=\lim_{u\uparrow t}\rho_{u-}(\omega)=\rho_{t-}(\omega),
\end{align}
where the order of limits can be swapped by monotonicity of the process and of the sequence. The convergence at $t=0$ is obvious as $\rho^n_{0-} = \rho_{0-} = 0$. This proves \eqref{eqn:lim-in-t-}.

Define for $t \in [0, T)$,
\begin{equation}\label{eqn:def_Rn}
R^n_t=\int_{[0, t]} (1-\chi_{s-}) \ud\rho^n_s, \qquad R_t=\int_{[0, t]} (1-\chi_{s}) \ud\rho_s,
\end{equation}
and extend both processes to $t=T$ in a continuous way by taking $R^n_{T}:=R^n_{T-}$ and $R_{T}:=R_{T-}$. By construction we have $(R^n)_{n\ge 1}\subset\tilde{\mcalAcirc} (\mcalG_t)$ and $R\in \tilde{\mcalAcirc}(\mcalG_t)$ and the idea is to apply Proposition \ref{prop:r-convergence}. First we notice that for all $\omega\in\Omega$ and any $t\in[0,T)$ we have
\[
\Delta R_t(\omega)=(1-\chi_t(\omega))\Delta\rho_t(\omega),
\] 
so that we can write the set of points of continuity of $R$ as (recall Definition \ref{def:def_C})
\[
C_R(\omega)=C_{\rho}(\omega)\cup \{t\in[0,T]:\chi_t(\omega)=1\}.
\] 

For any $t\in[0, T)$ and all $\omega\in\Omega_0$, integrating $R^n_t(\omega)$ by parts (\citep[Prop. 4.5, Chapter 0]{revuzyor}) and then taking limits as $n\to\infty$ we get
\begin{align}\label{eq:Rcon}
\lim_{n\to\infty}R^n_t(\omega)=&\,\lim_{n\to\infty} \Big[(1-\chi_{t}(\omega))\rho^n_t(\omega) - \int_{[0, t]} \rho^n_s(\omega)\ud (1-\chi_{s}(\omega))\Big]\\
=&\, (1-\chi_{t}(\omega))\rho_{t-}(\omega) - \int_{[0, t]} \rho_{s-}(\omega)\ud (1-\chi_{s}(\omega)) \notag\\
=&\, R_t(\omega) - (1-\chi_t(\omega)) \Delta \rho_t(\omega)=R_{t-}(\omega),\notag 
\end{align}
where the second equality uses dominated convergence and \eqref{eq:conv-rho}, and the third equality is integration by parts. We can therefore conclude that 
\begin{align*}
\lim_{n\to\infty}R^n_t(\omega)=R_t(\omega),\quad\text{for all $t\in C_R(\omega) \cap[0,T)$ and all $\omega\in\Omega_0$}.
\end{align*}

It remains to show the convergence at $T$ which is in $C_R(\omega)$ by our construction of $R$. Since the function $t \mapsto \rho_t(\omega)$ is non-decreasing and the sequence $(\rho^n(\omega))_n$ is non-decreasing, the sequence $(R^n(\omega))_n$ is non-decreasing too (an easy proof of this fact involves integration by parts and observing that $t\mapsto \ud (1-\chi_t(\omega))$ defines a negative measure; notice also the link to the first-order stochastic dominance). As in \eqref{eq:lim-nt}, we show that $\lim_{n \to \infty} R^n_{T-}(\omega) = R_{T-} (\omega)$ for $\omega \in \Omega_0$. By construction of $R^n$ and $R$, this proves convergence of $R^n_T$ to $R_T$.

Then, the processes $R^n$ and $R$ fulfil all the assumptions of Proposition \ref{prop:r-convergence} whose application allows us to obtain \eqref{eq:lim-in-A}.
\end{proof}

From the convergence \eqref{eq:Rcon}, an identical argument as in \eqref{eq:lim-nt} proves convergence of left-limits of processes $(R^n)$ at any $t \in [0, T]$. The following corollary formalises this observation. It will be used in Section \ref{sec:relax}.
\begin{cor}\label{cor:lim_R_in-t-}
Consider the processes $(R^n)$ and $R$ defined in \eqref{eqn:def_Rn}. For $\prob$\ae $\omega \in \Omega$ we have
\begin{equation*}
\lim_{n\to\infty} R^n_{t-}(\omega)=R_{t-}(\omega)\quad \text{for all $t \in [0, T]$}.
\end{equation*}
\end{cor}


\subsection{Verification of the conditions of Sion's theorem}\label{sec:verif}

For the application of Sion's theorem, we will consider a weak topology on $\laac(\ef^1_t)$ and $\laa(\ef^2_t)$ inherited from the space $\es$. In our arguments, we will often use that for convex sets the weak and strong closedness are equivalent \cite[Theorem 3.7]{Brezis2010} (although weak and strong convergence are not equivalent, c.f. \cite[Corollary 3.8]{Brezis2010}).

\begin{Lemma} \label{lem-strat-set-compact}
For any filtration $(\mcalG_t) \subseteq (\ef_t)$ satisfying the usual conditions, the set $\laa(\mcalG_t)$ is weakly compact in $\es$.
\end{Lemma}
\begin{proof}
We write $\laa$ for $\laa(\mcalG_t)$ and $\mcalAcirc$ for $\mcalAcirc(\mcalG_t)$. The set $\laa$ is a subset of a ball in $\es$. Since $\es$ is a reflexive Banach space, this ball is weakly compact (Kakutani's theorem, \cite[Theorem 3.17]{Brezis2010}). Therefore, we only need to show that $\laa$ is weakly closed. Since $\laa$ is convex, it is enough to show that $\laa$ is strongly closed \cite[Theorem 3.7]{Brezis2010}.

Take a sequence $(\rho^n)_{n\ge 1}\subset\laa$ that converges strongly in $\es$ to $\rho$. We will prove that $\rho\in\laa$ by constructing a \cadlag non-decreasing adapted process $(\hat \rho_t)$ such that $\hat\rho_{0-} = 0$, $\hat \rho_T = 1$, and $\hat \rho = \rho$ $(\lambda\times\prob)$\ae With no loss of generality we can pass to the \cadlag representatives $(\hat \rho^n)_{n\ge 1}\subset\mcalAcirc$ which also converge to $\rho$ in $\mathcal S$. Then, there is a subsequence $(n_k)_{k\ge 1}$ such that $\hat\rho^{n_k}\to \rho$ $(\lambda \times \prob)$\ae \cite[Theorem 4.9]{Brezis2010}.

Since 
\[
\int_0^t \prob\big(\lim_{k\to\infty}\hat\rho^{n_k}_s=\rho_s\big) \ud s=t,\quad\text{for all $t\in[0,T]$,}
\]
we can find $\hat D\subset [0,T]$ with $\lambda([0,T]\setminus\hat D)=0$ such that $\mathbb{P}(\Omega_t)=1$ for all $t\in \hat D$, where 
\[
\Omega_t:=\{\omega\in\Omega: \lim_{k\to\infty}\hat \rho^{n_k}_t(\omega)=\rho_t(\omega)\}.
\]
Then we can take a dense countable subset $D\subset \hat D$ and define $\Omega_0:=\cap_{t\in D}\Omega_t$ so that $\mathbb{P}(\Omega_0)=1$ and \[
\lim_{k\to\infty}\hat \rho^{n_k}_t(\omega)=\rho_t(\omega),\qquad\text{for all $(t,\omega)\in D\times\Omega_0$.}
\]
Since $\hat \rho^{n_k}$ are non-decreasing, so is the mapping $D\ni t\mapsto \rho_t(\omega)$ for all $\omega\in\Omega_0$. Let us extend this mapping to $[0,T]$ by defining $\hat \rho_t(\omega):=\rho_t(\omega)$ for $t\in D$ and
\[
\hat{\rho}_t(\omega):=\lim_{s\in D:s\downarrow t} \rho_s(\omega),\quad\hat{\rho}_{0-}(\omega):=0,\quad \hat{\rho}_{T}(\omega):=1,\quad\text{for all $\omega\in \Omega_0$,}
\]
where the limit exists due to monotonicity. For $\omega\in \mathcal{N}:=\Omega\setminus\Omega_0$, we set $\hat \rho_t(\omega) = 0$ for $t < T$ and $\hat \rho_T(\omega)=1$.  Notice that $\mathcal{N}\in\mcalG_0$ since $\prob(\mathcal{N})=0$ so that $\hat{\rho}_t$ is $\mcalG_t$-measurable for $t\in D$. Moreover, $\hat{\rho}$ is \cadlag by construction and $\hat{\rho}_t$ is measurable with respect to $\cap_{s\in D, s > t\,}\mcalG_s=\mcalG_{t+}=\mcalG_t$ for each $t\in[0,T]$ by the right-continuity of the filtration.  Hence, $\hat \rho$ is $(\mcalG_t)$-adapted and $\hat \rho \in \mcalAcirc$.

It remains to show that $\hat \rho^{n_k}\to \hat{\rho}$ in $\es$ so that $\hat \rho=\rho$ $(\lambda\times\prob)$\ae and therefore $\rho\in\mathcal{A}$. It suffices to show that $\hat \rho^{n_k}\to \hat{\rho}$ $(\lambda\times\prob)$\ae and then conclude by dominated convergence that $\hat \rho^{n_k}\to \hat{\rho}$ in $\es$. For each $\omega\in\Omega_0$ the process $t\mapsto\hat \rho(\omega)$ has at most countably many jumps (on any bounded interval) by monotonicity, i.e., $\lambda([0,T]\setminus C_{\hat \rho}(\omega))=0$ (recall Definition \ref{def:def_C}). Moreover, arguing as in the proof of Lemma \ref{lem:cadlag_convergence}, we conclude
\[
\lim_{k\to\infty}\hat \rho^{n_k}_t(\omega)=\hat \rho_t(\omega),\quad\text{for all $t\in C_{\hat \rho}(\omega)$ and all $\omega\in\Omega_0$}.
\]
Since $(\lambda\!\times\!\mathbb{P})(\{(t,\omega)\!:\!t\in C_{\hat\rho}(\omega)\cap B,\omega\in\Omega_0\})\!=\!\lambda(B)$ for any bounded interval $B\subseteq[0,T]$ then $\hat \rho^{n_k}\!\to\! \hat{\rho}$ in $\es$ and $\laa$ is strongly closed in $\es$.
\end{proof}
 
\begin{remark}\label{rem-TV-norm-doesnt-work}
Our space $\laa(\mcalG_t)$ is the space of processes that generate randomised stopping times and for any $\rho\in\laa(\mcalG_t)$ we require that $\rho_T(\omega)=1$, for all $\omega\in\Omega$. In the finite horizon problem, i.e., $T<\infty$, such specification imposes a constraint that prevents a  direct use of the topology induced by the norm considered in \cite{TouziVieille2002}. Indeed, in \cite{TouziVieille2002} the space $\mathcal S$ is that of $(\mcalG_t)$-adapted processes $\rho$ with
\begin{equation*}
\|\rho\|^2:=\ee\bigg[\int_{0}^T (\rho_t)^2 \ud t + (\Delta \rho_T)^2\bigg] < \infty,\quad \Delta\rho_T:=\rho_T-\liminf_{t\uparrow T}\rho_t.
\end{equation*}
The space of generating processes $\laa(\mcalG_t)$ is not closed in the topology induced by $\|\cdot\|$ above: define a sequence $(\rho^n)_{n\ge 1}\subset\laa(\mcalG_t)$ by
\begin{equation*}
\rho^n_t = n \bigg(t - T + \frac{1}{n}\bigg)^+, \qquad t \in [0, T].
\end{equation*}
Then $\|\rho^n\|\to 0$ as $n\to\infty$ but $\rho\equiv 0\notin\laa(\mcalG_t)$ since it fails to be equal to one at $T$ (and it is not possible to select a representative from $\laa(\mcalG_t)$ with the equivalence relation induced by $\|\,\cdot\,\|$). 
\end{remark}

It is of interest to explore the relationship between the topology on $\laa(\mcalG_t)$ implied by the weak topology on $\es$ (denote it by $\mcalO_2$) and the topology introduced in \cite{BaxterChacon, Meyer} (denote it by $\mcalO_1$). The topology $\mcalO_1$ is the coarsest topology in which all functionals of the form
\begin{equation}\label{eqn:top_O2}
\laa(\mcalG_t) \ni \rho \mapsto \ee \Big[\int_{[0, T]} X_t\, \ud \rho_t\Big]
\end{equation}
are continuous for any $X \in \mcalL$ with continuous trajectories. Our topology $\mcalO_2$, instead, is the restriction to $\laa(\mcalG_t)$ of the weak topology on $\es$. That is, $\mcalO_2$ is the coarsest topology for which all functionals of the form
\begin{equation*}
\laa(\mcalG_t) \ni \rho \mapsto \ee \Big[\int_{[0, T]} \rho_t\, Y_t\, \ud t\Big]
\end{equation*}
are continuous for all $Y \in \es$. 
\begin{Lemma}\label{lem:top}
Topologies $\mcalO_1$ and $\mcalO_2$ are identical.
\end{Lemma}
\begin{proof}
Denoting 
\begin{align}\label{eq:Xint}
X_t = \int_{[0, t]} Y_t\, \ud t
\end{align} 
and integrating by parts, we obtain for $\rho \in \laa (\mcalG_t)$
\[
\ee \Big[\int_{[0, T]} \rho_t\, Y_t\, \ud t\Big]
=
\ee \Big[X_T \rho_T - X_0 \rho_{0-} - \int_{[0, T]} X_t\, \ud \rho_t \Big]
=
\ee \Big[X_T - \int_{[0, T]} X_t\, \ud \rho_t \Big],
\]
where we used that $\rho_T = 1$ and $\rho_{0-} = 0$, $\prob$\as Hence, $\mcalO_2$ is the coarsest topology on $\laa (\mcalG_t)$ for which functionals \eqref{eqn:top_O2} are continuous for all processes $X$ defined as in \eqref{eq:Xint}. Since these processes $X$ are continuous, we conclude that $\mcalO_2 \subset \mcalO_1$.

The set $\laa (\mcalG_t)$ is compact in the topologies $\mcalO_1$ \citep[Theorem 3]{Meyer} and $\mcalO_2$ (see Lemma \ref{lem-strat-set-compact} above). The compact Hausdorff topology is the coarsest among Hausdorff topologies \cite[Cor.~3.1.14, p. 126]{Engelking}. Since $\mcalO_2$ is Hausdorff by \cite[Prop~3.3]{Brezis2010}, so is $\mcalO_1$ and we have $\mcalO_1 = \mcalO_2$.

\end{proof}

\begin{remark}
Meyer \cite[Thm.~4]{Meyer} shows that if $\ef$ is separable (i.e., countably generated) then the topology $\mcalO_1$ (hence $\mcalO_2$) is metrizable. This could also be seen directly for the topology $\mcalO_2$ by \cite[Thm.~3.29]{Brezis2010}, because $\laa(\mcalG_t)$ is bounded in $\es$ and $\mcalO_2$ is the restriction to $\laa(\mcalG_t)$ of the weak topology on $\es$. Indeed, it emerges from this argument for the metrizability of $\mcalO_2$ that it is sufficient to require that only $\mcalG_T$ be separable.
\end{remark}

\begin{Lemma}\label{lem:semi-cont}
Given any $(\xi,\zeta)\in\laac(\ef^1_t) \times\laa (\ef^2_t)$, the functionals $N(\xi,\cdot):\laa(\ef^2_t)\to\er$ and $N(\cdot,\zeta):\laac(\ef^1_t)\to\er$ are, respectively, upper semicontinuous and lower semicontinuous in the strong topology of $\es$.
\end{Lemma}
\begin{proof}
Since $\xi \in \laac (\ef^1_t)$, we have from \eqref{eq-functional-in-terms-of-controls} that the contribution of simultaneous jumps reduces to a single term:
\begin{equation}\label{eqn:N_cont}
N(\xi,\zeta)= \ee\bigg[\int_{[0, T)} f_t(1-\zeta_{t})d\xi_t + \int_{[0, T)} g_t(1-\xi_t)d\zeta_t + h_T\Delta\xi_T\Delta\zeta_T\bigg].
\end{equation}

\emph{Upper semicontinuity of $N(\xi,\cdot)$}. Fix $\xi\in\laac(\ef^1_t)$ and consider a sequence $(\zeta^{n})_{n \ge 1}\subset\laa(\ef^2_t)$ converging to $\zeta \in \laa(\ef^2_t)$ strongly in $\es$. We have to show that
\begin{equation*}
\limsup_{n\to\infty} N(\xi,\zeta^{n})\les N(\xi,\zeta).
\end{equation*}
Assume, by contradiction, that $\limsup_{n\to\infty} N(\xi,\zeta^{n}) > N(\xi,\zeta)$. There is a subsequence $(n_k)$ over which the limit on the left-hand side is attained. Along a further subsequence we have $(\prob\times\lambda)\ae$ convergence of $\zeta^n$ to $\zeta$ \citep[Theorem 4.9]{Brezis2010}. With an abuse of notation we will assume that the original sequence possesses those two properties, i.e., the limit $\lim_{n\to\infty} N(\xi,\zeta^{n})$ exists, it strictly dominates $N(\xi,\zeta)$, and there is $(\prob \times \lambda)\ae$ convergence of $\zeta^{n}$ to $\zeta$.

Since $\xi$ is absolutely continuous on $[0, T)$, 
\begin{equation*}
\lim_{n\to\infty}\ee\bigg[\int_{[0, T)} f_t(1-\zeta^{n}_{t})\ud\xi_t\bigg]=\ee\bigg[\int_{[0, T)} f_t(1-\zeta_{t})\ud\xi_t \bigg]
\end{equation*}
by the dominated convergence theorem. For the last two terms of $N(\xi, \zeta^n)$ in \eqref{eqn:N_cont} we have
\begin{align*}
\ee\bigg[\int_{[0, T)} g_t(1-\xi_t)\ud\zeta^n_t + h_T\Delta\xi_T\Delta\zeta^n_T\bigg]
&=
\ee\bigg[\int_{[0, T)} g_t(1-\xi_{t-})\ud\zeta^n_t + h_T\Delta\xi_T\Delta\zeta^n_T\bigg]\\
&=
\ee\bigg[\int_{[0, T]} g_t(1-\xi_{t-})\ud\zeta^n_t + (h_T - g_T) \Delta\xi_T\Delta\zeta^n_T\bigg],
\end{align*}
where the first equality is by the continuity of $\xi$ and for the second one we use that $1-\xi_{T-} = \Delta \xi_T$. From Lemma \ref{lem:cadlag_convergence} and the boundedness and continuity of $(\xi_t)$ we verify the assumptions of Proposition \ref{prop:r-convergence} (with $X_t=g_t(1-\xi_{t-})$ therein since $\xi_{t-}$ is continuous on $[0,T]$), so
\begin{equation*}
\lim_{n \to \infty} \ee\bigg[\int_{[0, T]} g_t(1-\xi_{t-})\ud\zeta^n_t \bigg] = \ee\bigg[\int_{[0, T]} g_t(1-\xi_{t-})\ud\zeta_t \bigg].
\end{equation*}
Recalling that $g_T \le h_T$, we obtain from Lemma \ref{prop-terminal-time-jump-limit}
\begin{equation*}
\limsup_{n\to\infty} \ee\big[(h_T-g_T)\Delta\xi_T\Delta\zeta^n_T\big] \les \ee\big[(h_T-g_T)\Delta\xi_T\Delta\zeta_T\big].
\end{equation*}
Combining above convergence results contradicts $\lim_{n\to\infty} N(\xi,\zeta^n) >N(\xi,\zeta)$, hence, proves the upper semicontinuity.

\emph{Lower semicontinuity of $N(\cdot,\zeta)$}. Fix $\zeta\in\laa(\ef^2_t)$ and consider a sequence $(\xi^{n})_{n\ge 1}\subset\laac(\ef^1_t)$ converging to $\xi \in \laac (\ef^1_t)$ strongly in $\es$. Arguing by contradiction as above, we assume that there is a subsequence of $\xi^{n}$, which we denote the same, such that $\xi^{n} \to \xi$ $(\prob\times\lambda)$\ae and 
\begin{equation}\label{eqn:two_terms2}
\lim_{n\to\infty} N(\xi^{n},\zeta) < N(\xi,\zeta). 
\end{equation}
By Lemma \ref{lem:cadlag_convergence} and the continuity of $(\xi_t)$ we have for $\prob$\ae $\omega\in\Omega$
\[
\lim_{n\to\infty} \xi^{n}_t(\omega) = \xi_t(\omega)\quad\text{for all $t \in [0,T)$}.
\] 
Then by dominated convergence for the second term of $N(\xi^n,\zeta)$ in \eqref{eqn:N_cont} we get
\[
\lim_{n\to\infty}\ee\bigg[\int_{[0, T)} g_t(1-\xi^{n}_t)\ud\zeta_t\bigg]= \ee\bigg[\int_{[0, T)} g_t(1-\xi_t)\ud\zeta_t\bigg].
\]
For the remaining terms of $N(\xi^{n}, \zeta)$, we have 
\[
\ee\bigg[\int_{[0, T)} f_t(1-\zeta_{t})\ud\xi^n_t + h_T\Delta\xi^n_T\Delta\zeta_T\bigg]=\ee\bigg[\int_{[0, T)} f_t(1-\zeta_{t})\ud\xi^n_t + f_T \Delta\xi^n_T\Delta\zeta_T + (h_T-f_T)\Delta\xi^n_T\Delta\zeta_T\bigg]. 
\]
Observe that, by Lemma \ref{prop-terminal-time-jump-limit},
\begin{equation*}
\liminf_{n\to\infty}\ee\big[(h_T-f_T)\Delta\xi^n_T\Delta\zeta_T\big]\ges \ee\big[(h_T-f_T)\Delta\xi_T\Delta\zeta_T\big],
\end{equation*}
because $h_T-f_T\les 0$. Further,
\begin{equation*}
\lim_{n\to\infty}\ee\bigg[\int_{[0, T)} f_t(1-\zeta_{t})\ud\xi^n_t + f_T \Delta\xi^n_T\Delta\zeta_T \bigg]=\ee\bigg[\int_{[0, T)} f_t(1-\zeta_{t})\ud\xi_t + f_T \Delta\xi_T\Delta\zeta_T\bigg]
\end{equation*}
by Corollary \ref{cor-specific-convergence-2}. The above results contradict \eqref{eqn:two_terms2}, therefore, proving the lower semicontinuity.
\end{proof}

We are now ready to prove that the game with continuous randomisation for the first player ($\tau$-player) has a value.
\begin{proof}[{\bf Proof of Theorem \ref{th-value-cont-strat}}]
We will show that the conditions of Sion's theorem hold (recall the notation in Theorem \ref{th-the-Sion}) with $(A,B)=(\laa(\ef^2_t),\laac(\ef^1_t))$ on the space $\es \times \es$ equipped with its weak topology. For the sake of compactness of notation, we will write $\laa$ for $\laa(\ef^2_t)$ and $\laac$ for $\laac(\ef^1_t)$. It is straightforward to verify that the sets $\laa$ and $\laac$ are convex. Compactness of $\laa$ in the weak topology of $\es$ follows from Lemma \ref{lem-strat-set-compact}. It remains to prove the convexity and semi-continuity properties of $N$ with respect to the weak topology of $\es$. This is equivalent to showing that for any $a\in\er$, $\hat{\xi}\in\laac$ and $\hat{\zeta}\in\laa$ the level sets 
\[
\mcal{K}(\hat{\zeta},a)=\{\xi\in\laac:N(\xi,\hat{\zeta})\les a\} \qquad \text{and}\qquad  \mcal{Z}(\hat{\xi},a)=\{\zeta\in\laa:N(\hat{\xi},\zeta)\ges a\}
\]
are convex and closed in $\laac$ and $\laa$, respectively, with respect to the weak topology of $\es$. For any $\lambda \in [0,1]$ and $\xi^{1}, \xi^{2} \in \laac$, $\zeta^{1}, \zeta^{2} \in \laa$, using the expression in \eqref{eq-functional-in-terms-of-controls} it is immediate (by linearity) that 
\begin{align*}
N(\lambda \xi^{1} + (1-\lambda) \xi^{2}, \hat \zeta) &= \lambda N(\xi^{1}, \hat \zeta) + (1-\lambda) N(\xi^{2}, \hat\zeta),\\
N(\hat\xi, \lambda \zeta^{1} + (1-\lambda)\zeta^{2}) &= \lambda N(\hat\xi, \zeta^{1}) + (1-\lambda) N(\hat\xi, \zeta^{2}).
\end{align*}
This proves the convexity of the level sets. Their closedness in the strong topology of $\es$ is established in Lemma \ref{lem:semi-cont}. The latter two properties imply, by \cite[Theorem 3.7]{Brezis2010}, that the level sets are closed in the weak topology of $\es$. Sion's theorem (Theorem \ref{th-the-Sion}) yields the existence of the value of the game: $W_* = W^*$.

The second part of the statement results from using a version of Sion's theorem proved in \cite{Komiya1988} which allows to write $\max$ instead of $\sup$ in \eqref{eq-value-cont-restriction}, i.e.,
\begin{equation*}
\sup_{\zeta\in\laa}\inf_{\xi\in\laac} N(\xi,\zeta)=\max_{\zeta\in\laa}\inf_{\xi\in\laac} N(\xi,\zeta)=\inf_{\xi\in\laac} N(\xi,\zeta^*),
\end{equation*}
where $\zeta^*\in\laa$ delivers the maximum.
\end{proof}


\subsection{Approximation with continuous controls}\label{sec:approx}
We now prove Proposition \ref{thm:conv_lipsch} by constructing a sequence $(\xi^{n})$ of Lipschitz continuous processes with the Lipschitz constant for each process bounded by $n$ for all $\omega$. This uniform bound on the Lipschitz constant is not used in this paper as we only need that each of the processes $(\xi^{n}_t)$ has absolutely continuous trajectories with respect to the Lebesgue measure on $[0,T)$ so that it belongs to $\laa_{ac}(\ef^1_t)$.

\begin{proof}[Proof of Proposition \ref{thm:conv_lipsch}]

Fix $\zeta\in\laa(\ef^2_t)$. We need to show that for any $\xi\in\laa(\ef^1_t)$, there exists a sequence $(\xi^{n})_{n \ge 1} \subset \laac(\ef^1_t)$ such that
\begin{equation}
\limsup_{n\to\infty} N(\xi^{n},\zeta)\les N(\xi,\zeta).
\label{eq-liminf-M}
\end{equation}

We will explicitly construct absolutely continuous $\xi^{n}$ that approximate $\xi$ in a suitable sense. As $N(\xi, \zeta)$ does not depend on the choice of  \cadlag representatives, by Definition \ref{def:integral}, without loss of generality we assume that $\xi \in \mcalAcirc(\ef^1_t)$ and $\zeta \in \mcalAcirc(\ef^2_t)$. Define a function $\phi^n_t = (nt)\wedge 1\vee 0$. Let $\xi^{n}_t = \int_{[0, t]} \phi^n_{t-s} \ud\xi_s$ for $t\in[0,T)$, and $\xi^{n}_T = 1$. We shall show that $(\xi^{n}_t)$ is $n$-Lipschitz, hence absolutely continuous on $[0, T)$.
Note that $\phi^n_t\equiv 0$ for $t\les 0$, and therefore $\xi^{n}_t = \int_{[0, T]} \phi^n_{t-s} \ud\xi_s$ for $t \in [0, T)$. For arbitrary $t_1,t_2\in[0,T)$ we have
\begin{align*}
|\xi^{n}_{t_1}-\xi^{n}_{t_2}| 
&= 
\left|\int_{[0, T]} (\phi^n_{t_1-s}-\phi^n_{t_2-s}) \ud\xi_s\right| 
\les
\int_{[0, T]} |\phi^n_{t_1-s}-\phi^n_{t_2-s}| \ud\xi_s\\
&\les 
\int_{[0, T]} n|(t_1-s)-(t_2-s)| \ud\xi_s 
=
\int_{[0, T]} n|t_1-t_2| d\xi_s=n|t_1-t_2|,
\end{align*}
where the first inequality is Jensen's inequality (which is applicable since $\xi(\omega)$ is a cumulative distribution function on $[0, T]$ for each $\omega$), and the second inequality follows by the definition of $\phi^n$.

We will verify the assumptions of Proposition \ref{prop-specific-convergence-3}. Clearly the sequence $(\xi^n)$ is non-decreasing in $n$, as the measure $\ud \xi(\omega)$ is positive for each $\omega \in \Omega$ and the sequence $\phi^n$ is non-decreasing. By the construction of $\xi^{n}$ we have $\xi^{n}_0 = 0 \to \xi_{0-}$ as $n\to\infty$. Moreover, for any $t \in (0, T)$ and $n > 1/t$
\begin{align*}
\xi^{n}_t = \int_{[0, t)}\phi^n_{t-s}\ud\xi_s=\xi_{t-\tfrac{1}{n}}+\int_{(t-\tfrac{1}{n}, t)}n(t-s)\ud\xi_s,
\end{align*}
where the first equality uses that $\phi^n_{0}=0$, so that jumps of $\xi$ at time $t$ give zero contribution, and the second one uses the definition of $\phi^n$. Letting $n\to\infty$ we obtain $\xi^{n}_t\to \xi_{t-}$ as the second term above vanishes since
\begin{align*}
0\le \int_{(t-\tfrac{1}{n}, t)}n(t-s)\ud\xi_s\le \xi_{t-} - \xi_{t-\tfrac{1}{n}}\to 0. 
\end{align*}
The continuity of $\xi^n$ on $[0, T)$ and Proposition \ref{prop-specific-convergence-3} imply that
\begin{equation*}
\lim_{n\to\infty}\ee\bigg[\int_{[0, T)}f_t(1-\zeta_{t})\ud\xi^{n}_t\bigg]
=
\lim_{n\to\infty}\ee\bigg[\int_{[0, T)}f_t(1-\zeta_{t-})\ud\xi^{n}_t\bigg]
= 
\ee\bigg[\int_{[0, T)}f_t(1-\zeta_{t})\ud\xi_t\bigg],
\end{equation*}
and $\lim_{n \to \infty} \xi^{n}_{T-} = \xi_{T-}$ so that
\[
\lim_{n\to\infty}\Delta\xi^{n}_T =\Delta\xi_T,
\]
since $\xi^{n}_T=1$ for all $n\ge 1$.
The dominated convergence theorem (applied to the second integral below) also yields
\begin{equation}\label{eqn:M_ij_conv}
\begin{aligned}
\lim_{n\to\infty} N(\xi^{n},\zeta)
&=\lim_{n\to\infty}\ee\bigg[\int_{[0, T)}f_{t}(1-\zeta_{t})\ud\xi^{n}_{t}+ \int_{[0, T)} g_t(1-\xi^{n}_{t})\ud\zeta_t + h_T\Delta\xi^{n}_T\Delta\zeta_T\bigg]\\
&= \ee\bigg[\int_{[0, T)}f_{t}(1-\zeta_{t})\ud\xi_{t} + \int_{[0, T)} g_t(1-\xi_{t-})\ud\zeta_t + h_T\Delta\xi_T\Delta\zeta_T\bigg].
\end{aligned}
\end{equation}
Note that
\begin{align}\label{eq-remove-common-jumps}
N(\xi,\zeta)&=\ee\bigg[\!\int_{[0, T)}\! f_t(1-\zeta_{t})\ud\xi_t +\! \int_{[0, T)}\! g_t(1\!-\xi_t)\ud\zeta_t + \sum_{t \in [0, T]} h_t \Delta\xi_t\Delta\zeta_t\bigg]\notag\\
&= \ee\bigg[\!\int_{[0, T)}\! f_t(1-\zeta_{t})\ud\xi_t +\! \int_{[0, T)}\! g_t(1-\xi_{t-})\ud\zeta_t
+\! \sum_{t \in [0, T)}\!\! (h_t-g_t)\Delta\xi_t\Delta\zeta_t + h_T\Delta\xi_T\Delta\zeta_T\bigg]\\
&\ges \ee\bigg[\!\int_{[0, T)}\! f_t(1-\zeta_{t})\ud\xi_t +\! \int_{[0, T)}\! g_t(1-\xi_{t-})\ud\zeta_t + h_T\Delta\xi_T\Delta\zeta_T\bigg],\notag
\end{align}
where the last inequality is due to Assumption \ref{eq-order-cond}. Combining this with \eqref{eqn:M_ij_conv} completes the proof of \eqref{eq-liminf-M}.
\end{proof}

\subsection{Relaxation of Assumption \ref{ass:regular}}\label{sec:relax}

Assumption \ref{ass:regular} which requires that the payoff processes be regular can be relaxed to allow for a class of jumps including predictable ones with nonzero conditional mean (i.e., violating regularity, see Eq. \eqref{eq:cond-reg}). In this section we extend Theorem \ref{th-value-cont-strat} and Proposition \ref{thm:conv_lipsch} to the case of Assumption \ref{ass:regular_gen} with $(\hat g_t)$ from the decomposition of the payoff process $g$ being non-decreasing. In this case we must `smoothen' the generating process $\xi$ of the minimiser in order to guarantee the desired semi-continuity properties of the game's expected payoff (see Remark \ref{rem:contrad}). Arguments when $(\hat f_t)$ from the decomposition of $f$ in Assumption \ref{ass:regular_gen} is non-increasing are analogous thanks to the symmetry of the set-up pointed out in Remark \ref{rem:ineq}. However, in that case we restrict strategies of the maximiser to absolutely continuous generating processes $\zeta\in\laac(\ef^2_t)$ and the first player (minimiser) picks $\xi\in\laa(\ef^1_t)$.

\begin{theorem}\label{th-value-cont-strat_gen}
Under assumptions \ref{eq-integrability-cond}, \ref{ass:regular_gen}, \ref{eq-order-cond}-\ref{ass:filtration}, (with $\hat g$ non-decreasing) the game \eqref{eq-value-cont-restriction} has a value, i.e.
\begin{equation*}
W_{*}=W^{*}:=W.
\end{equation*}
Moreover, the $\zeta$-player (maximiser) has an optimal strategy, i.e. there exists $\zeta^*\in\laa(\ef^2_t)$ such that
\begin{equation*}
\inf_{\xi\in\laac(\ef^1_t)} N(\xi,\zeta^*)=W.
\end{equation*}
\end{theorem}
\begin{Proposition}\label{thm:conv_lipsch_gen}
Under assumptions \ref{eq-integrability-cond}, \ref{ass:regular_gen}, \ref{eq-order-cond}-\ref{ass:filtration}, (with $\hat g$ non-decreasing) for any $\zeta \in \laa(\ef^2_t)$ and $\xi \in \laa(\ef^1_t)$, there is a sequence $\xi^n \in \laac(\ef^1_t)$ such that
\[
\limsup_{n \to \infty} N(\xi^n, \zeta) \le N(\xi, \zeta).
\]
\end{Proposition}
\begin{proof}[{\bf Proof of Theorem \ref{thm:main2}}]
The proof of the existence of the value is identical to the proof Theorem \ref{thm:main} but with references to Theorem \ref{th-value-cont-strat} and Proposition \ref{thm:conv_lipsch} replaced by the above results. 

For the existence of the saddle point, the additional requirement that $\hat g$ be non-decreasing {\em and} $\hat f$ be non-increasing guarantees the complete symmetry of the problem when swapping the roles of the two players as in Remark \ref{rem:ineq}. Thus, the same proof as in Theorem \ref{thm:main} can be repeated verbatim.
\end{proof}

In the rest of the section we prove Theorem \ref{th-value-cont-strat_gen} and Proposition \ref{thm:conv_lipsch_gen}. Processes $\hat f, \hat g$ have the following decomposition according to Theorem VI.52 in \cite{DellacherieMeyer} and remarks thereafter: there are $(\ef_t)$-stopping times $(\eta^f_k)_{k \ge 1}$ and $(\eta^g_k)_{k \ge 1}$, non-negative $\ef_{\eta^f_k}$-measurable random variables $X^f_k$, $k \ge 1$, and non-negative $\ef_{\eta^g_k}$-measurable random variables $X^g_k$, $k \ge 1$, such that 
\begin{equation}\label{eqn:decomposition_piecewise}
\hat f_t = \sum_{k=1}^\infty (-1)^k X^f_k \ind{\{t \ge \eta^f_k\}}, \qquad \hat g_t = \sum_{k=1}^\infty X^g_k \ind{\{t \ge \eta^g_k\}}.
\end{equation}
The alternating terms in the sum for $(\hat f_t)$ come from interweaving sequences for the two non-decreasing processes $(\hat f^+_t)$ and $(\hat f^-_t)$ from $\mcalL$ arising from the decomposition of the integrable variation process $(\hat f_t)$ (recall $\hat f_t=\hat f^+_t-\hat f^-_t$). This is for notational convenience and resulting in no mathematical complications as the infinite sum is absolutely convergent. Recall that $\hat g$ is assumed non-decreasing.

The condition that $\hat f_0 = \hat g_0 = 0$ means that $\eta^f_k, \eta^g_k > 0$ for all $k \ge 1$. Since $\hat f, \hat g$ have integrable variation (in the sense of \cite[p. 115]{DellacherieMeyer}), the infinite sequences in \eqref{eqn:decomposition_piecewise} are dominated by integrable random variables $X^f$ and $X^g$: for any $t \in [0, T]$
\begin{align}\label{eq:Xfg}
|\hat f_t| \le X^f := \sum_{k=1}^\infty X^f_k, \qquad \text{and}\qquad \hat g_t \le X^g := \sum_{k=1}^\infty X^g_k.
\end{align}

To handle convergence of integrals with piecewise-constant processes, we need to extend the results of Proposition \ref{prop:r-convergence}.
\begin{Proposition}\label{prop:A}
For a filtration $(\mcalG_t)\subseteq (\ef_t)$, consider $(\rho^n)_{n \ge 1} \subset \tl\mcalAcirc(\mcalG_t)$ and $\rho \in \tl\mcalAcirc(\mcalG_t)$ with
\[
\prob\Big(\Big\{\omega \in \Omega:\ \lim_{n\to\infty}\rho^n_t(\omega)=\rho_t(\omega),\quad\text{for all $t\in C_\rho(\omega)\cup\{T\}$} \Big\} \Big) = 1.
\]
Then for any $\ef$-measurable random variables $\theta\in(0,T]$ and $X\in[0,\infty)$  with $\ee[X] < \infty$ we have
\begin{equation}\label{eqn:theta_t}
\limsup_{n\to\infty}\ee\Big[\int_{[0, T]} \ind{\{t \ge \theta\}} X \ud\rho^n_t\Big] \le \ee\Big[\int_{[0, T]} \ind{\{t \ge \theta\}} X \ud\rho_t\Big].
\end{equation}
Furthermore, if $\prob (\{\omega:\ \theta(\omega) \in C_\rho(\omega) \text{ or } X(\omega) = 0\}) = 1$, then
\begin{equation}\label{eqn:theta_t_eq}
\lim_{n\to\infty}\ee\Big[\int_{[0, T]} \ind{\{t \ge \theta\}} X \ud\rho^n_t\Big] = \ee\Big[\int_{[0, T]} \ind{\{t \ge \theta\}} X \ud\rho_t\Big].
\end{equation}
\end{Proposition}
\begin{proof}
Let $\Omega_0$ be the set of $\omega \in \Omega$ for which $\rho^n_t(\omega) \to \rho_t(\omega)$ for all $t \in C_\rho(\omega) \cup \{T\}$. 
Fix $\omega \in \Omega_0$. For any $t$ such that $t \in C_\rho(\omega)$ and $t < \theta(\omega)$ (such $t$ always exists as $\theta(\omega) > 0$ and $\rho$ has at most countably many jumps on any bounded interval) we have $\rho^n_t(\omega) \le \rho^n_{\theta(\omega)-}(\omega)$ so that by assumption 
\[
\liminf_{n \to \infty} \rho^n_{\theta(\omega)-} (\omega) \ge \rho_{t} (\omega).
\]
Since $C_{\rho}(\omega)$ is dense in $(0, T)$, by arbitrariness of $t<\theta(\omega)$ we have
\begin{equation}\label{eqn:hash}
\liminf_{n \to \infty} \rho^n_{\theta(\omega)-} (\omega) \ge \rho_{\theta(\omega)-} (\omega).
\end{equation}
We rewrite the integral as follows: $\int_{[0, T]} \ind{\{t \ge \theta\}} X \ud\rho^n_t = X (\rho^n_T - \rho^n_{\theta-})$. Therefore,
\begin{align*}
\limsup_{n\to\infty}\ee\Big[\int_{[0, T]} \ind{\{t \ge \theta\}} X \ud\rho^n_t\Big]
=
\limsup_{n\to\infty}\ee\big[X (\rho^n_T - \rho^n_{\theta-}) \big]
=
\limsup_{n \to \infty} \ee [X \rho^n_T] - \liminf_{n \to \infty} \ee[X \rho^n_{\theta-}].
\end{align*}
The dominated convergence theorem yields that $\lim_{n \to \infty} \ee [X \rho^n_T] = \ee [X \rho_T]$, while applying Fubini's theorem gives
\[
\liminf_{n \to \infty} \ee[X \rho^n_{\theta-}] \ge  \ee[\liminf_{n \to \infty} X \rho^n_{\theta-}] \ge \ee[ X \rho_{\theta-}],
\]
where the last inequality is by \eqref{eqn:hash}. Combining the above estimates completes the proof of \eqref{eqn:theta_t}.

Assume now that $\theta(\omega) \in C_\rho(\omega)$ or $X(\omega) = 0$ for $\prob$\ae $\omega \in \Omega_0$. This and the dominated convergence theorem yield
\[
\ee[X(\rho_T - \rho_{\theta-})] = \ee[X(\rho_T - \rho_{\theta})] = \lim_{n \to \infty} \ee[X(\rho^n_T - \rho^n_{\theta})] \le \limsup_{n \to \infty} \ee[X(\rho^n_T - \rho^n_{\theta-})],
\]
where the last inequality follows from the monotonicity of $\rho^n$. This estimate and \eqref{eqn:theta_t} prove \eqref{eqn:theta_t_eq}.
\end{proof}

\begin{remark}\label{rem:contrad0}
The inequality \eqref{eqn:theta_t} in Proposition \ref{prop:A} can be strict even if $\rho^n_t \to \rho_t$ for all $t \in [0, T]$ because this condition does not imply that $\rho^n_{t-} \to \rho_{t-}$. One needs further continuity assumptions on $(\rho_t)$ to establish equality \eqref{eqn:theta_t_eq}.
\end{remark}

\begin{proof}[{\bf Proof of Theorem \ref{th-value-cont-strat_gen}}]
Compared to the proof of the analogue result under the more stringent condition \ref{ass:regular} (i.e., Theorem \ref{th-value-cont-strat}), we only need to establish lower and upper semicontinuity of the functional $N$, while all other remaining arguments stay valid. For the semicontinuity, we extend arguments of Lemma \ref{lem:semi-cont}. 

\emph{Upper semicontinuity of $N(\xi,\cdot)$}. Fix $\xi\in\laac(\ef^1_t)$ and consider a sequence $(\zeta^{n})_{n \ge 1}\in\laa(\ef^2_t)$ converging to $\zeta \in \laa(\ef^2_t)$ strongly in $\es$. Arguing by contradiction, we assume that there is a subsequence of $(\zeta^n)_{n \ge 1}$ denoted the same with an abuse of notation, that converges $(\prob \times \lambda)\ae$ to $\zeta$ and such that
\begin{equation*}
\lim_{n\to\infty} N(\xi,\zeta^{n}) > N(\xi,\zeta).
\end{equation*}
Without loss of generality, we can further require that $(\zeta^n)_{n \ge 1} \subset \mcalAcirc(\ef^2_t)$ and $\zeta \in \mcalAcirc(\ef^2_t)$. 
Since $\xi$ is absolutely continuous on $[0, T)$, 
\begin{equation}
\lim_{n\to\infty}\ee\bigg[\int_{[0, T)} f_t(1-\zeta^{n}_{t})\ud\xi_t\bigg]=\ee\bigg[\int_{[0, T)} f_t(1-\zeta_{t})\ud\xi_t \bigg]
\label{eq-int-conv-1a}
\end{equation}
by the dominated convergence theorem. For the last two terms of $N(\xi, \zeta^n)$ (recall \eqref{eqn:N_cont}) we have
\begin{align*}
\ee\bigg[\int_{[0, T)} g_t(1-\xi_t)\ud\zeta^n_t + h_T\Delta\xi_T\Delta\zeta^n_T\bigg]
&=
\ee\bigg[\int_{[0, T]} g_t(1-\xi_{t-})\ud\zeta^n_t + (h_T - g_T) \Delta\xi_T\Delta\zeta^n_T\bigg].
\end{align*}
As in the proof of Lemma \ref{lem:semi-cont}, for the regular part $\tl g$ of the process $g$ we have
\begin{equation}\label{eqn:tl_g_conv}
\lim_{n \to \infty} \ee\bigg[\int_{[0, T]} \tl g_t(1-\xi_{t-})\ud\zeta^n_t \bigg] = \ee\bigg[\int_{[0, T]} \tl g_t(1-\xi_{t-})\ud\zeta_t \bigg].
\end{equation}

For the pure jump part $\hat g$ of the process $g$, we will prove that 
\begin{equation}\label{eqn:hat_g_conv}
\limsup_{n \to \infty} \ee\bigg[\int_{[0, T]} \hat g_t(1-\xi_{t-})\ud\zeta^n_t \bigg] \le \ee\bigg[\int_{[0, T]} \hat g_t(1-\xi_{t-})\ud\zeta_t \bigg].
\end{equation}
To this end, let us define
\[
R^n_t = \int_{[0, t]} (1-\xi_{s-})\ud\zeta^n_s, \qquad R_t = \int_{[0, t]} (1-\xi_{s-})\ud\zeta_s, \qquad \text{for $t\in[0,T]$,}
\]
with $R^n_{0-} = R_{0-} = 0$ and then we are going to apply Proposition \ref{prop:A} with $R^n$ and $R$ instead of $\rho^n$ and $\rho$. We need $R^n_t (\omega) \to R_t(\omega)$ as $n\to\infty$ for $t \in C_R(\omega)=C_{\zeta}(\omega)\cup \{t\in[0,T]:\xi_t(\omega)=1\}$, for $\prob$\ae $\omega \in \Omega$. 
The latter is indeed true. Setting $\Omega_0 = \{\omega \in \Omega:\ \lim_{n \to \infty} \zeta^n_t(\omega) = \zeta_t(\omega)\ \forall\, t \in C_\zeta(\omega) \}$, we have $\prob(\Omega_0) = 1$ by Lemma \ref{lem:cadlag_convergence}. For any $\omega \in \Omega_0$ and $t \in C_\zeta(\omega)$, invoking the absolute continuity of $(\xi_t)$, we obtain (omitting the dependence on $\omega$)
\begin{equation*}
\lim_{n \to \infty} R^n_t 
=
\lim_{n \to \infty} \Big[ (1-\xi_t) \zeta^n_t + \int_{[0, t]} \zeta^n_s \ud \xi_s \Big]
=
(1-\xi_t) \zeta_t + \int_{[0, t]} \zeta_s \ud \xi_s = R_t,
\end{equation*}
where the convergence of the second term is the consequence of the dominated convergence theorem and the fact that $\lambda ([0,T]\setminus C_\zeta(\omega)) = 0$ and $\zeta^n_T = \zeta_T = 1$.

For any $k \ge 1$, since $X^g_k \ge 0$, Proposition \ref{prop:A} gives (recall \eqref{eqn:decomposition_piecewise})
\begin{equation}\label{eqn:limsup01}
\limsup_{n \to \infty} \ee \bigg[ \int_{[0, T]} X^g_k \ind{\{t \ge \eta^g_k\}} \ud R^n_t \bigg] \le \ee \bigg[\int_{[0, T]} X^g_k \ind{\{t \ge \eta^g_k\}} \ud R_t \bigg].
\end{equation}
We apply the decomposition of $\hat g$ and then the monotone convergence theorem
\[
\ee\bigg[\int_{[0, T]} \hat g_t(1-\xi_{t-})\ud\zeta^n_t \bigg] = \ee \bigg[ \sum_{k=1}^\infty \int_{[0, T]} X^g_k \ind{\{t \ge \eta^g_k\}} \ud R^n_t \bigg]
= \sum_{k=1}^\infty \ee \bigg[ \int_{[0, T]} X^g_k \ind{\{t \ge \eta^g_k\}} \ud R^n_t \bigg].
\]
Since $\hat g \in \mcalL$ we have the bound (recall \eqref{eq:Xfg})
\[
\sum_{k=1}^\infty \sup_n \ee \bigg[ \int_{[0, T]} X^g_k \ind{\{t \ge \eta^g_k\}} \ud R^n_t \bigg] \le \sum_{k=1}^\infty \ee [ X^g_k ] 
< \infty .
\]
Then we can apply (reverse) Fatou's lemma (with respect to the counting measure on $\mathbb{N}$)
\begin{align*}
\limsup_{n \to \infty} \sum_{k=1}^\infty \ee \bigg[ \int_{[0, T]} X^g_k \ind{\{t \ge \eta^g_k\}} \ud R^n_t \bigg]
&\le
\sum_{k=1}^\infty \limsup_{n \to \infty} \ee \bigg[ \int_{[0, T]} X^g_k \ind{\{t \ge \eta^g_k\}} \ud R^n_t \bigg]\\
&\le
\sum_{k=1}^\infty \ee \bigg[\int_{[0, T]} X^g_k \ind{\{t \ge \eta^g_k\}} \ud R_t \bigg]
=
\ee\bigg[\int_{[0, T]} \hat g_t(1-\xi_{t-})\ud\zeta_t \bigg],
\end{align*}
where the last inequality is due to \eqref{eqn:limsup01} and the final equality follows by monotone convergence and the decomposition of $\hat g$. This completes the proof of \eqref{eqn:hat_g_conv}.

Recalling that $g_T \le h_T$, we obtain from Lemma \ref{prop-terminal-time-jump-limit}
\begin{equation*}
\limsup_{n\to\infty} \ee\big[(h_T-g_T)\Delta\xi_T\Delta\zeta^n_T\big] \les \ee\big[(h_T-g_T)\Delta\xi_T\Delta\zeta_T\big],
\end{equation*}
and combining the latter with \eqref{eqn:tl_g_conv}, \eqref{eqn:hat_g_conv} and \eqref{eq-int-conv-1a} shows that
\begin{align}\label{eq:usc_gen}
\limsup_{n \to \infty} N(\xi, \zeta^n) \le N(\xi, \zeta).
\end{align}
Hence we have a contradiction with $\lim_{n\to\infty} N(\xi,\zeta^n) >N(\xi,\zeta)$, which proves the upper semicontinuity.

\emph{Lower semicontinuity of $N(\cdot,\zeta)$}. The proof follows closely the argument of the proof of Lemma \ref{lem:semi-cont}: we fix $\zeta\in\laa(\ef^2_t)$, consider a sequence $(\xi^{n})_{n\ge 1}\subset\laac(\ef^1_t)$ converging to $\xi \in \laac (\ef^1_t)$ strongly in $\es$, assume that \eqref{eqn:two_terms2} holds and reach a contradiction. We only show how to handle the convergence for $(\hat f_t)$ as all other terms are handled by the proof of Lemma \ref{lem:semi-cont}.

By Lemma \ref{lem:cadlag_convergence} and the continuity of $(\xi_t)$ we have $\prob \big( \lim_{n\to\infty} \xi^{n}_t(\omega) = \xi_t(\omega)\ \forall\,t \in [0,T)\big) = 1$. 
Let
\[
R^n_t = \int_{[0,t]} (1-\zeta_{t-})\ud\xi^n_t, \qquad R_t = \int_{[0,t]} (1-\zeta_{t-})\ud\xi_t,
\]
with $R^n_{0-} = R_{0-} = 0$. Due to the continuity of $(\xi^n_t)$ and $(\xi_t)$ for $t \in [0, T)$, processes $(R^n_t)$ and $(R_t)$ are continuous on $[0, T)$ with a possible jump at $T$. From \eqref{eqn:conv_R} in the proof of Proposition \ref{prop-specific-convergence-2} we conclude that for $\prob$\ae $\omega \in \Omega$
\[
\lim_{n \to \infty} R^n_t(\omega) = R_t(\omega) \quad \text{for all $t \in [0, T]$}.
\]
Since $\Delta \hat f_T = 0$ (see Assumption \ref{ass:regular_gen}), there is a decomposition such that $X^f_k \ind{\{\eta^f_k = T\}} = 0$ $\prob$\as for all $k$. Recalling that $(R_t)$ is continuous on $[0, T)$, we can apply \eqref{eqn:theta_t_eq} in Proposition \ref{prop:A}: for any $k \ge 1$
\[
\lim_{n \to \infty} \ee \bigg[ \int_{[0, T]} X^f_k \ind{\{t \ge \eta^f_k\}} \ud R^n_t \bigg] = \ee \bigg[\int_{[0, T]} X^f_k \ind{\{t \ge \eta^f_k\}} \ud R_t \bigg].
\]
Combining the latter with decomposition \eqref{eqn:decomposition_piecewise} and the dominated convergence theorem (with the bound $X^f$) we obtain
\[
\lim_{n \to \infty} \ee \bigg[ \int_{[0, T]} \hat f_t \ud R^n_t \bigg] = \ee \bigg[\int_{[0, T]} \hat f_t \ud R_t \bigg].
\]
Arguing as in the proof of Corollary \ref{cor-specific-convergence-2}, we have
\begin{equation*}
\lim_{n \to \infty} \ee \bigg[ \int_{[0, T)} \hat f_t (1-\zeta_t) \ud \xi^n_t + \hat f_T \Delta \zeta_T \Delta \xi^n_T\bigg] = \ee \bigg[\int_{[0, T)} \hat f_t (1-\zeta_t) \ud \xi_t + \hat f_T \Delta \zeta_T \Delta \xi_T \bigg].
\end{equation*}
Corollary \ref{cor-specific-convergence-2} implies an analogous convergence for $(\tl f_t)$ and the rest of the proof of lower semicontinuity from Lemma \ref{lem:semi-cont} applies.
\end{proof}

\begin{remark}\label{rem:contrad}
In the arguments above, item (4) in Assumption \ref{ass:regular_gen} implies in particular that the payoff process $(g_t)$ does not have predictable jumps that are $\prob$\as negative. This assumption cannot be further relaxed as this may cause the proof of the upper semicontinuity in Theorem \ref{th-value-cont-strat_gen} to fail. Recall that the process $(g_t)$ corresponds to the payoff of the second player and her strategy $(\zeta_t)$ is not required to be absolutely continuous. For example, fix $t_0\in(0,T)$ and take $g_t=1-\ind{\{t\ge t_0\}}$, $\zeta_t = \ind{\{t\ges t_0\}}$ and $\xi_t = \ind{\{t = T\}}$. Let us consider the sequence $\zeta^n_t =\ind{\{t\ges t_0-\frac{1}{n}\}}$, which converges to $\zeta$ pointwise and also strongly in $\es$. We have
\begin{equation*}
\int_{[0, T]} g_t(1-\xi_{t-})\ud\zeta^n_t\equiv 1, \:\:\: \text{for all $n$'s, but}\:\:\: \int_{[0, T]} g_t(1-\xi_{t-})\ud\zeta_t\equiv 0,
\end{equation*}
hence \eqref{eqn:hat_g_conv} fails and so does \eqref{eq:usc_gen}.
\end{remark}

\begin{proof}[{\bf Proof of Proposition \ref{thm:conv_lipsch_gen}}]
Here, we also only show how to extend the proof of Proposition \ref{thm:conv_lipsch} to the more general setting. Fix $\zeta \in \mcalAcirc(\ef^2_t)$ and $\xi \in \mcalAcirc(\ef^1_t)$. Construct a sequence $(\xi^n) \subset \mcalAcircAc(\ef^1_t)$ as in the proof of Proposition \ref{thm:conv_lipsch}. It is sufficient to show that
\begin{equation}\label{eqn:lipch_conv_liminf}
\limsup_{n \to \infty} N(\xi^n, \zeta) \le N(\xi, \zeta).
\end{equation}

From the proof of Proposition \ref{thm:conv_lipsch} we have that
\begin{equation}\label{eqn:M_ij_conv_gen}
\begin{aligned}
&\lim_{n\to\infty}\ee\bigg[\int_{[0, T)}\tl f_{t}(1-\zeta_{t})\ud\xi^{n}_{t}+ \int_{[0, T)} \tl g_t(1-\xi^{n}_{t})\ud\zeta_t + h_T\Delta\xi^{n}_T\Delta\zeta_T\bigg]\\
&= \ee\bigg[\int_{[0, T)}\tl f_{t}(1-\zeta_{t})\ud\xi_{t} + \int_{[0, T)} \tl g_t(1-\xi_{t-})\ud\zeta_t + h_T\Delta\xi_T\Delta\zeta_T\bigg].
\end{aligned}
\end{equation}
For $t \in [0, T]$, define
\[
R^n_t = \int_{[0, t]} (1-\zeta_{s-})\ud\xi^n_s, \qquad R_t = \int_{[0, t]} (1-\zeta_{s})\ud\xi_s
\]
with $R^n_{0-} = R_{0-} = 0$. Corollary \ref{cor:lim_R_in-t-} implies that for $\prob$\ae $\omega \in \Omega$
\begin{align}\label{eq:convR}
\lim_{n \to \infty} R^n_{t-}(\omega) = R_{t-}(\omega) \quad \text{for all $t \in [0, T]$}.
\end{align}
By the decomposition of $(\hat f_t)$ in \eqref{eqn:decomposition_piecewise} and the dominated convergence theorem for the infinite sum (recalling \eqref{eq:Xfg}) we obtain
\begin{align*}
\ee\bigg[\int_{[0, T)}\hat f_{t}(1-\zeta_{t})\ud\xi^{n}_{t}\bigg] 
&= 
\ee\bigg[\int_{[0, T)}\hat f_{t}(1-\zeta_{t-})\ud\xi^{n}_{t}\bigg]
=
\sum_{k=1}^\infty \ee\bigg[(-1)^k \int_{[0, T)} X^f_k \ind{\{t \ge \eta^f_k\}} \ud R^n_t \bigg]\\
&=
\sum_{k=1}^\infty \ee \big[ (-1)^k X^f_k (R^n_{T-} - R^n_{\eta^f_k-})\big],
\end{align*}
where the first equality follows from the continuity of $(\xi^n_t)$ on $[0, T)$.
We further apply dominated convergence (with respect to the product of the counting measure on $\mathbb N$ and to the measure $\prob$) to obtain
\begin{equation}\label{eqn:part2}
\begin{aligned}
\lim_{n\to\infty}\ee\bigg[\!\int_{[0, T)}\!\hat f_{t}(1-\zeta_{t})\ud\xi^{n}_{t}\bigg] 
&=\sum_{k=1}^\infty  \ee \big[(-1)^k \lim_{n\to\infty}  X^f_k (R^n_{T-} - R^n_{\eta^f_k-})\big]\\
&=
\sum_{k=1}^\infty \ee \big[ (-1)^k X^f_k (R_{T-} - R_{\eta^f_k-})\big]=
\ee\bigg[\int_{[0, T)}\hat f_{t}(1-\zeta_{t})\ud\xi_{t}\bigg],
\end{aligned}
\end{equation}
where the second equality uses \eqref{eq:convR} and the final one the decomposition of $\hat f$. 
Recalling that $\xi^n_t\to\xi_{t-}$ as $n\to\infty$ by construction, dominated convergence gives
\begin{equation}\label{eqn:part3}
\lim_{n\to\infty}\ee\bigg[\int_{[0, T)}\hat g_{t}(1-\xi^n_{t})\ud\zeta_{t}\bigg]
=
\ee\bigg[\int_{[0, T)}\hat g_{t}(1-\xi_{t-})\ud\zeta_{t}\bigg].
\end{equation}
Putting together \eqref{eqn:M_ij_conv_gen}, \eqref{eqn:part2} and \eqref{eqn:part3} shows
\[
\lim_{n \to \infty} N(\xi^n, \zeta) = \ee\bigg[\int_{[0, T)} f_{t}(1-\zeta_{t})\ud\xi_{t} + \int_{[0, T)} g_t(1-\xi_{t-})\ud\zeta_t + h_T\Delta\xi_T\Delta\zeta_T\bigg].
\]
It remains to notice that by \eqref{eq-remove-common-jumps} the right hand side is dominated by $N(\xi, \zeta)$, which completes the proof of \eqref{eqn:lipch_conv_liminf}.
\end{proof}

\subsection{Proof of Theorem \ref{thm:ef_0_value}} \label{sec:ef_functional}
Randomisation devices $Z_\tau$ and $Z_\sigma$ associated to a pair $(\tau,\sigma)\in\te^R(\ef^1_t)\times\te^R(\ef^2_t)$ are independent of $\mcalG$. Denoting by $(\xi_t) \in \mcalAcirc(\ef^1_t)$ and $(\zeta_t) \in \mcalAcirc(\ef^2_t)$ the generating processes for $\tau$ and $\sigma$, respectively, the statement of Proposition \ref{prop-functionals-equal} can be extended to encompass the conditional functional \eqref{eqn:cond_func}:
\begin{equation}\label{eqn:cond_reform}
\ee\big[ \mcal{P}(\tau, \sigma) \big| \mcalG \big] = \ee\bigg[\int_{[0, T)} f_{t}(1-\zeta_{t})\ud \xi_t +  \int_{[0, T)} g_{t}(1-\xi_t) \ud \zeta_t + \sum_{t \in [0, T]} h_t \Delta \xi_t \Delta \zeta_t \bigg|\mcalG\bigg].
\end{equation}
We can also repeat the same argument as in Remark \ref{rem-Laraki-Solan} to obtain that
\[
\underline V:=\esssup_{\sigma\in\te^R(\ef^2_t)}\essinf_{\tau \in \te^R(\ef^1_t)}  \ee\big[ \mcal{P}(\tau, \sigma) \big| \mcalG \big]=\esssup_{\sigma\in\te^R(\ef^2_t)}\essinf_{\tau \in \te(\ef^1_t)}  \ee\big[ \mcal{P}(\tau, \sigma) \big| \mcalG \big]
\]
and 
\[
\overline V:=\essinf_{\tau \in \te^R(\ef^1_t)} \esssup_{\sigma\in\te^R(\ef^2_t)} \ee\big[ \mcal{P}(\tau, \sigma) \big| \mcalG \big]=\essinf_{\tau \in \te^R(\ef^1_t)}\esssup_{\sigma\in\te(\ef^2_t)}  \ee\big[ \mcal{P}(\tau, \sigma) \big| \mcalG \big].
\]
Notice that $\overline V\ge \underline V$, $\prob$\as We will show that
\begin{align}\label{eq:EV}
\ee[\,\underline V\,]=\ee[\,\overline V\,],
\end{align}
so that $\overline V= \underline V$\,, $\prob$\as as needed.

In order to prove \eqref{eq:EV}, let us define
\begin{equation*}
\overline{M}(\tau):=\esssup_{\sigma \in \te(\ef^2_t)}  \ee\big[ \mcal{P}(\tau, \sigma) \big| \mcalG \big],\quad\text{for $\tau\in\te^R(\ef^1_t)$},
\end{equation*}
and 
\begin{align*}
\underline{M}(\sigma):=\essinf_{\tau \in \te(\ef^1_t)}  \ee\big[ \mcal{P}(\tau, \sigma) \big| \mcalG \big],\quad\text{for $\sigma\in\te^R(\ef^2_t)$}.
\end{align*}
These are two standard optimal stopping problems and the theory of Snell envelope applies (see, e.g., \cite[Appendix D]{Karatzas1998} and \cite{elkaroui1981}). We adapt some results from that theory to suit our needs in the game setting.
\begin{Lemma}\label{lem:directed}
The family $\{\overline{M}(\tau),\,\tau\in \te^R(\ef^1_t)\}$ is downward directed and the family $\{\underline{M}(\sigma),\,\sigma\in \te^R(\ef^2_t)\}$ is upward directed. 
\end{Lemma}
\begin{proof}
Let $\tau^{(1)},\tau^{(2)}\in\te^R(\ef^1_t)$ and let $\xi^{(1)},\xi^{(2)}\in\mcalAcirc(\ef^1_t)$ be the corresponding generating processes. Fix the $\mcalG$-measurable event $B=\{\overline{M}(\tau^{(1)})\les\overline{M}(\tau^{(2)})\}$ and define another $(\ef^1_t)$-randomised stopping time as $\hat{\tau}=\tau^{(1)} \ind{B} +\tau^{(2)} \ind{B^c}$. We use $\mcalG\subset\ef^1_0$ to ensure that $\hat \tau\in\te^R(\ef^1_t)$. The generating process of $\hat \tau$ reads $\hat \xi_t=\xi^{(1)}_t \ind{B} +\xi^{(2)}_t \ind{B^c}$ for $t\in[0,T]$. Using the linear structure of $\hat \xi$ and recalling \eqref{eqn:cond_reform}, for any $\sigma\in\te(\ef^2_t)$, we have
\begin{align*}
\ee\big[\mcal{P}(\hat{\tau}, \sigma)|\mcalG\big]
&=
\ind{B}\ee\bigg[\int_{[0, \sigma)} f_{u}\ud \xi^{(1)}_u +  g_{\sigma}(1-\xi^{(1)}_\sigma) + h_\sigma \Delta \xi^{(1)}_\sigma \bigg|\mcalG\bigg]\\
&\hspace{12pt}+\ind{B^c}\ee\bigg[\int_{[0, \sigma)} f_{u}\ud \xi^{(2)}_u +  g_{\sigma}(1-\xi^{(2)}_\sigma) + h_\sigma \Delta \xi^{(2)}_\sigma\bigg|\mcalG\bigg]\\
&=\ind{B}\ee\big[\mcal{P}(\tau^{(1)} ,\sigma)|\mcalG\big]+\ind{B^c}\ee\big[\mcal{P}(\tau^{(2)} ,\sigma)|\mcalG\big]\\
&\le\ind{B}\overline{M}(\tau^{(1)})+\ind{B^c}\overline{M}(\tau^{(2)})=\overline{M}(\tau^{(1)})\wedge\overline{M}(\tau^{(2)}),
\end{align*} 
where the inequality is by definition of essential supremum and the final equality by definition of the event $B$.
Thus, taking the supremum over $\sigma\in\te(\ef^2_t)$ we get
\[
\overline{M}(\hat \tau)\le \overline{M}(\tau^{(1)})\wedge\overline{M}(\tau^{(2)}),
\]
hence the family $\{\overline{M}(\tau),\,\tau\in \te^R(\ef^1_t)\}$ is downward directed. A symmetric argument proves that 
the family $\{\underline{M}(\sigma),\,\sigma\in \te^R(\ef^2_t)\}$ is upward directed.
\end{proof}
An immediate consequence of the lemma  and of the definition of essential supremum/infimum is that (see, e.g., \cite[Lemma I.1.3]{Peskir2006}) we can find sequences $(\sigma_n)_{n\ge 1}\subset \te^R(\ef^2_t)$ and $(\tau_n)_{n\ge 1}\subset\te^R(\ef^1_t)$ such that $\prob$\as
\begin{align}\label{eq:limV}
\overline V=\lim_{n\to\infty}\overline{M}(\tau_n)\quad\text{and}\quad\underline V=\lim_{n\to\infty}\underline{M}(\sigma_n),
\end{align}
where the convergence is monotone in both cases. 

Analogous results hold for the optimisation problems defining $\overline M(\tau)$ and $\underline M(\sigma)$. The proof of the following lemma is similar to that of Lemma \ref{lem:directed} and omitted.
\begin{Lemma}\label{lem:directed2}
The family $\{\ee\big[\mcal{P}(\tau,\sigma)|\mcalG\big],\,\sigma\in \te(\ef^2_t)\}$ is upward directed for each $\tau\in \te^R(\ef^1_t)$. The family $\{\ee\big[\mcal{P}(\tau,\sigma)|\mcalG\big],\,\tau\in \te(\ef^1_t)\}$ is downward directed for each $\sigma\in \te^R(\ef^2_t)$. 
\end{Lemma}
It follows that for each $\tau\in\te^R(\ef^1_t)$ and $\sigma\in\te^R(\ef^2_t)$, there are sequences $(\sigma^\tau_n)_{n\ge 1}\subset\te(\ef^2_t)$ and $(\tau^\sigma_n)_{n\ge 1}\subset\te(\ef^1_t)$ such that 
\begin{align}\label{eq:limM}
\overline M(\tau)=\lim_{n\to\infty}\ee\big[\mcal{P}(\tau,\sigma^\tau_n)|\mcalG\big]\quad\text{and}\quad\underline M(\sigma)=\lim_{n\to\infty}\ee\big[\mcal{P}(\tau^\sigma_n,\sigma)|\mcalG\big],
\end{align}
where the convergence is monotone in both cases. Equipped with these results we can prove the following lemma which will quickly lead to \eqref{eq:EV}.

\begin{Lemma}\label{cor-exp-for-values}
Recall $V_*$ and $V^*$ as in Definition \ref{def-value-rand-strat}. We have
\begin{equation}\label{eq-cor-exp}
\ee[\overline{V}] = V^*, \qquad \text{and}\qquad \ee[\underline{V}] = V_*.
\end{equation}
\end{Lemma}
\begin{proof}
Fix $\tau\in\te^R(\ef^1_t)$. By \eqref{eq:limM} and the monotone convergence theorem
\[
\ee[ \overline{M}(\tau) ] = \lim_{n\to\infty}\ee[\mcal{P}(\tau,\sigma^\tau_n)]\le \sup_{\sigma\in\te(\ef^2_t)}\ee[\mcal{P}(\tau,\sigma)].
\]
The opposite inequality follows from the fact that $\overline{M}(\tau) \ge \ee[\mcal{P}(\tau,\sigma)|\mcalG]$ for any $\sigma \in \te(\ef^2_t)$ by the definition of the essential supremum. Therefore, we have
\begin{equation}\label{eqn:M1}
\ee[ \overline{M}(\tau) ] = \sup_{\sigma\in\te(\ef^2_t)}\ee[\mcal{P}(\tau,\sigma)]. 
\end{equation}
From \eqref{eq:limV}, similar arguments as above prove that
\begin{equation}\label{eqn:M2}
\ee[\overline{V}] = \inf_{\tau \in \te^R(\ef^1_t)} \ee[ \overline{M}(\tau) ].
\end{equation}
Combining \eqref{eqn:M1} and \eqref{eqn:M2} completes the proof that $\ee[\overline{V}] = V^*$. The second part of the statement requires analogous arguments.
\end{proof}

Finally, \eqref{eq-cor-exp} and Theorem \ref{thm:main2} imply \eqref{eq:EV}, which concludes the proof of Theorem \ref{thm:ef_0_value}.

%
%

\section{Counterexamples}
\label{sec:Nikita-examples}

In the three subsections below we show that: (a) relaxing condition \ref{eq-order-cond} may lead to a game without a value, (b) in situations where one player has all the informational advantage, the use of randomised stopping times may still be beneficial also for the uninformed player, and (c) Assumption \ref{ass:regular_gen} is tight in requiring that either $(\hat f_t)$ is non-increasing or $(\hat g_t)$ is non-decreasing. 

In order to keep the exposition simple we consider the framework of Section \ref{subsec:game_1} with $\I = 2$, $\J = 1$, and impose that $(\ef^p_t)$ be the trivial filtration (hence all payoff processes are deterministic, since they are $(\ef^p_t)$-adapted). Furthermore we restrict our attention to the case in which $f^{1,1} = f^{2,1} = f$, $g^{1,1} = g^{2,1} = g$ and $h^{1,1}_t\ind{\{t<T\}}=h^{2,1}_t\ind{\{t<T\}}=f_t\ind{\{t<T\}}$. Only the terminal payoff depends on the scenario, i.e., $h^{1,1}_T\neq h^{2,1}_T$ (both deterministic). For notational simplicity we set $h^{1}:=h^{1,1}_T$ and $h^{2}:=h^{2,1}_T$.

Notice that only the first player (minimiser) observes the true value of $\mcalI$, so she has a strict informational advantage over the second player (maximiser).
The second player will be referred to as the \emph{uninformed player} while the first player as the \emph{informed player}.

We denote by $\te^R$ the set of $(\ef^p_t)$-randomised stopping times. The informed player chooses two randomised stopping times $\tau_1, \tau_2$ (one for each scenario, recall Lemma \ref{lem:tau_decomposition}) with the generating processes $\xi^1,\xi^2$ which, due to the triviality of the filtration $(\ef_t^p)$, are deterministic functions. Pure stopping times are constants in $[0,T]$. Similarly, the uninformed player's randomised stopping time $\sigma$ has the generating process $\zeta$ that is a deterministic function.

\subsection{A game without a value when \ref{eq-order-cond} fails}
Let us consider specific payoff functions 
\begin{equation*}
f \equiv 1,\quad g_t=\frac{1}{2}t, \quad h^1= 2,\quad h^2=0,
\end{equation*}
and let us also set $T=1$, $\pi_1 = \pi_2 =\frac{1}{2}$.

\begin{Proposition}
In the example of this subsection we have
\begin{equation*}
V_{*}\les\frac{1}{2} \qquad \text{and}\qquad V^* > \frac12,
\end{equation*}
so the game does not have a value.
\end{Proposition}
\begin{proof}
First we show that $V_{*}\les\frac{1}{2}$. Recall that (c.f. Remark \ref{rem-Laraki-Solan})
\begin{equation*}
V_*=\sup_{\sigma\in\te^R}\inf_{\tau_1,\tau_2\in\te^R}N((\tau_1,\tau_2),\sigma)=\sup_{\sigma\in\te^R}\inf_{\tau_1,\tau_2\in[0,1]}N((\tau_1,\tau_2),\sigma),
\end{equation*} 
so we can take $\tau_1, \tau_2 \in [0, 1]$ deterministic in the arguments below. Take any $\sigma \in \te^R$ and the corresponding generating process $(\zeta_t)$ which is, due to the triviality of the filtration $(\ef^p_t)$,
a deterministic function.  For $\tau_1\in[0,1)$, $\tau_2=1$ we obtain
\begin{align*}
N((\tau_1,\tau_2),\sigma)&= \ee\big[(\frac{1}{2}\sigma \indd{\sigma<\tau_1}+1\cdot \ind{\{\sigma\ges\tau_1\}})\indd{\mcalI = 1} + (\frac{1}{2}\sigma \indd{\sigma<1}+0 \cdot \indd{\sigma=1}) \indd{\mcalI = 2}\big]\\
&\le \frac{1}{2}(\frac{1}{2}\zeta_{\tau_1-}+(1-\zeta_{\tau_1-}))+ \frac{1}{4}\zeta_{1-}
= \frac{1}{2}-\frac{1}{4}\zeta_{\tau_1-}+\frac{1}{4}\zeta_{1-},
\end{align*}
where we used that $\sigma$ is bounded above by $1$ and that $\mcalI$ is independent of $\sigma$ with $\prob(\mcalI=1)=\prob(\mcalI=2)=\tfrac{1}{2}$.
In particular,
\begin{equation*}
\inf_{\tau_1,\tau_2 \in [0, 1]} N((\tau_1,\tau_2),\sigma) \le \lim_{\tau_1\to 1-} N((\tau_1,1),\sigma) = \frac{1}{2}.
\end{equation*}
This proves that $V_* \le \frac12$.

Now we turn our attention to demonstrating that $V^{*}>\frac{1}{2}$. Noting again that
\begin{equation*}
V^*=\inf_{\tau_1,\tau_2\in\te^R}\sup_{\sigma\in\te^R}N(\tau_1,\tau_2,\sigma)=\inf_{\tau_1,\tau_2\in\te^R}\sup_{\sigma\in[0,1]}N(\tau_1,\tau_2,\sigma),
\end{equation*} 
we can restrict our attention to constant $\sigma \in [0, 1]$. Take any $\tau_1, \tau_2 \in \te^R$ and the corresponding generating processes $(\xi^1_t), (\xi^2_t)$ which are also deterministic functions. 

Take any $\delta \in (0, 1/2)$. If $\xi^1_{1-} > \delta$, then for any $\sigma<1$ we have
\begin{align*}
N((\tau_1,\tau_2),\sigma)
&\ge 
\ee\big[\big(1 \cdot \indd{\tau_1\les\sigma}+ \frac{1}{2}\sigma \indd{\sigma<\tau_1}\big) \indd{\mcalI=1}+\frac{1}{2}\sigma \indd{\mcalI=2}\big]\\
&=
\ee\big[\big(\xi^1_\sigma + \frac{1}{2}\sigma(1-\xi^1_\sigma)\big) \ind{\{\mcalI=1\}}+\frac{1}{2}\sigma \ind{\{\mcalI=2\}}\big]\\
&=
\frac{1}{2}\xi^1_\sigma - \frac{1}{4}\sigma\xi^1_\sigma+\frac{1}{2}\sigma
=
\frac{1}{2}\xi^1_\sigma(1 - \frac{1}{2}\sigma)+\frac{1}{2}\sigma,
\end{align*}
and, in particular,
\begin{equation*}
\sup_{\sigma \in [0, 1]} N((\tau_1,\tau_2),\sigma) \ge \lim_{\sigma\to 1-} N((\tau_1,\tau_2),\sigma) \ge \frac{1}{4}\xi^1_{1-} + \frac{1}{2}\ge \frac12 + \frac14\delta>\frac12.
\end{equation*}
On the other hand, if $\xi^1_{1-} \le \delta$, taking $\sigma=1$ yields
\begin{equation*}
\sup_{\sigma \in [0, 1]} N((\tau_1,\tau_2),\sigma) \ge N((\tau_1,\tau_2),1) \ge \ee[2\cdot \ind{\{\tau_1= 1\}} \ind{\{\mcalI=1\}}] = 1-\xi^1_{1-}  \ge 1 - \delta > \frac{1}{2}.
\end{equation*}
This completes the proof that $V^* > \frac12$.
\end{proof}


\subsection{Necessity of randomization}\label{sec-example-necessity-of-rand}

Here we argue that randomisation is not only sufficient in order to find the value in Dynkin games with asymmetric information but in many cases it is also necessary. In \cite{DEG2020} there is a rare example of explicit construction of optimal strategies for a zero-sum Dynkin game with asymmetric information in a diffusive set-up (see Section \ref{subsec:game_2} above for details). The peculiarity of the solution in \cite{DEG2020} lies in the fact that the informed player uses a randomised stopping time whereas the uninformed player sticks to a pure stopping time. An interpretation of that result suggests
that the informed player uses randomisation to `gradually reveal' information about the scenario in which the game is being played, in order to induce the uninformed player to act in a certain desirable way. Since the uninformed player has `nothing to reveal' one may be tempted to draw a general conclusion that she should never use randomised stopping rules. However, Proposition \ref{prop-uninf-benefits-from-rand} below shows that such conclusion would be wrong in general and even the \emph{uninformed} player may benefit from randomisation of stopping times.

\begin{figure}[tb]
\begin{center}
\includegraphics[width=0.6\textwidth]{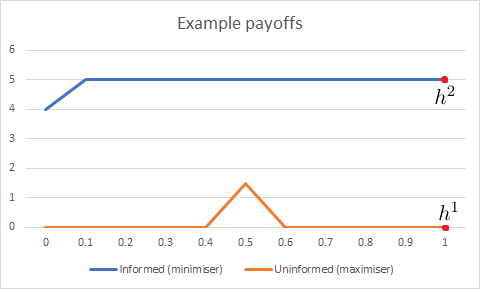}
\end{center}
\caption{Payoff functions $f$ in blue, $g$ in orange.}
\label{fig:1}
\end{figure}

We consider specific payoff functions $f$ and $g$ plotted on Figure \ref{fig:1}. Their analytic formulae read
\[
f_t = (10t+4)\ind{\{t\in[0,\frac{1}{10})\}} + 5\ind{\{t\in[\frac{1}{10},1]\}}, \qquad 
g_t = (15t - 6) \ind{\{t\in[\frac{2}{5},\frac{1}{2})\}}+(9-15 t)\ind{\{t\in[\frac{1}{2},\frac{3}{5})\}}
\]
with
\begin{equation*}
h^1 = 0 = g_{1-},\quad h^2=5=f_{1-}.
\end{equation*}
We also set $T=1$, $\pi_1 = \pi_2 =\frac{1}{2}$.
As above, we identify randomized strategies with their generating processes. In particular, we denote by $\zeta$ the generating process for $\sigma \in \te^R$. 

By Theorem \ref{thm:main},  the game has a value in randomised strategies, i.e., $V^* = V_*$. Restriction of the uninformed player's (player 2) strategies to pure stopping times affects only the lower value, see Remark \ref{rem-Laraki-Solan}. The lower value of the game in which player 2 is restricted to using pure stopping times reads
\begin{equation*}
\what{V}_*:=\sup_{\sigma\in[0,1]}\inf_{\tau_1, \tau_2 \in \te^R} N((\tau_1, \tau_2),\sigma)=\sup_{\sigma\in[0,1]}\inf_{\tau_1, \tau_2 \in [0,1]}  N((\tau_1, \tau_2),\sigma),
\end{equation*}
where the equality is again due to Remark \ref{rem-Laraki-Solan} (notice that here all pure stopping times are $(\ef^p_t)$-stopping times hence deterministic, because $(\ef^p_t)$ is trivial). As the following proposition shows, $\what{V}_*<V_*$, so the game in which the uninformed player does not randomise does not have a value. This confirms that the randomisation can play a strategic role beyond manipulating information.
\begin{Proposition}\label{prop-uninf-benefits-from-rand}
In the example of this subsection, we have
\begin{equation*}
V_*>\what{V}_*.
\end{equation*}
\end{Proposition}
\begin{proof}
First, notice that 
\begin{equation*}
\what{V}_*\les\sup_{\sigma\in [0,1]} N(\hat{\tau}(\sigma),\sigma),
\end{equation*}
where we take
\[
\hat \tau (\sigma)= (\tau_1(\sigma), \tau_2(\sigma)) = \begin{cases}
(1,1),& \text{for } \sigma\in[0,1),\\
(1,0),& \text{for } \sigma=1.
            \end{cases}
\]
It is easy to verify that $\sup_{\sigma\in[0,1]} N(\hat{\tau}(\sigma),\sigma)=2$. 

We will show that the $\sigma$-player can ensure a strictly larger payoff by using a randomised strategy. Define $\zeta_t=a \ind{\{t\ges\frac{1}{2}\}}+(1-a)\ind{\{t=1\}}$, i.e., the corresponding $\sigma\in\te^R$ prescribes to `stop at time $\frac{1}{2}$ with probability $a$ and at time $1$ with probability $1-a$'. The value of the parameter $a \in [0,1]$ will be determined below. We claim that
\begin{equation}\label{eqn:zb}
\inf_{\tau_1, \tau_2 \in[0,1]} N((\tau_1, \tau_2),\zeta) = N((1,0),{\zeta})\wedge N((1,1),{\zeta}).
\end{equation}
Assuming that the above is true, we calculate
\[
N((1,0),{\zeta})=2+\frac{3}{4}a, \qquad N((1,1),{\zeta}) = \frac{5}{2}-a.
\]
Picking $a = \frac27$ the above quantities are equal to $\frac{31}{14}$. Hence $V_* \ge \frac{31}{14}>2$.

It remains to prove \eqref{eqn:zb}. Recall that $\zeta_t=a \ind{\{t\ges\frac{1}{2}\}}+(1-a)\ind{\{t=1\}}$ is the generating process of $\sigma$ and the expected payoff reads
\begin{equation*}
N((\tau_1, \tau_2), \zeta) = \sum_{i=1}^2 \ee \big[ \ind{\{\mcalI=i\}}\left(f_{\tau_i} \ind{\{\tau_i \le \sigma\} \cap \{\tau_i < 1\}} + g_{\sigma} \ind{\{\sigma<\tau_i\} \cap \{\sigma < 1\}} + h^i \ind{\{\tau_i = \sigma = 1\}}\right) \big].
\end{equation*}
It is clear that on the event $\{\mcalI=1\}$ the infimum is attained for $\tau_1=1$, irrespective of the choice of $\zeta$. On the event $\{\mcalI=2\}$ the informed player would only stop either at time zero, where the function $f$ attains the minimum cost $f_0=4$, or at time $t>\frac12$ since  $\zeta$ only puts mass at $t=\frac12$ and at $t=1$ (the informed player knows her opponent may stop at $t=\frac12$ with probability $a$). The latter strategy corresponds to a payoff $5-\frac72 a$ and can also be achieved by picking $\tau_2=1$. Then the informed player needs only to consider the expected payoff associated to the strategies $(\tau_1,\tau_2)=(1,0)$ and $(\tau_1,\tau_2)=(1,1)$, so that \eqref{eqn:zb} holds.
\end{proof}

\subsection{Necessity of Assumption \ref{ass:regular_gen}} \label{subsec:example_jumps}

Our final counter-example shows that violating Assumption \ref{ass:regular_gen} by allowing both predictable upward jumps of $f$ \emph{and} predictable downward jumps of $g$ may also lead to a game without a value.

Consider the payoffs
\[
f_t=1+2\ind{\{t\ge \frac{1}{2}\}},\quad g_t=-\ind{\{t\ge \frac{1}{2}\}},\quad h^1=3,\quad h^2=-1,
\]
so that $h^1=f_{1-}$ and $h^2=g_{1-}$ and let us also set $T=1$, $\pi_1=\pi_2=\tfrac{1}{2}$. Assumption \ref{ass:regular_gen} is violated as $g$ has a predictable downward jump and $f$ has a predictable upward jump at time $t=\frac{1}{2}$.

\begin{Proposition}
In the example of this subsection we have
\begin{equation*}
V_{*}\les 0,\quad\text{and}\quad V^{*}>0,
\end{equation*}
so the game does not have a value.
\end{Proposition}
\begin{proof}
First we show that $V_{*}\le 0$. For this step, it is sufficient to restrict our attention to pure stopping times $\tau_1,\tau_2\in[0,1]$ for the informed player (c.f. Remark \ref{rem-Laraki-Solan}). Let $\sigma\in\te^R$ with a (deterministic) generating process $(\zeta_t)$ and fix $\ve\in (0, \frac12)$. For $\tau_1=\frac{1}{2}-\ve$ and $\tau_2=1$ we obtain
\begin{align*}
N((\tau_1,\tau_2),\sigma)&
=\ee\big[\indd{\mcalI=1}(0\cdot\indd{\sigma<\tau_1}+1\cdot\indd{\sigma\ges\tau_1}) + \indd{\mcalI=2}(0\cdot\indd{\sigma<\frac{1}{2}}-1\cdot\indd{\sigma\ges \frac{1}{2}})\big]\\
&=\frac12\big(1-\zeta_{(\frac{1}{2}-\ve)-}\big)-\frac12\big(1-\zeta_{\frac{1}{2}-}\big).
\end{align*}
Therefore, using that $(\zeta_t)$ has \cadlag trajectories we have
\begin{equation*}
\inf_{\tau_1,\tau_2\in[0,1]} N((\tau_1,\tau_2),\sigma) \les \lim_{\ve\to 0} \frac12\cdot(\zeta_{\frac{1}{2}-}-\zeta_{(\frac{1}{2}-\ve)-}) =0.
\end{equation*}
Since the result holds for all $\sigma\in\te^R$ we have $V_*\le 0$.

Next, we demonstrate that $V^{*}>0$. For this step it is sufficient to consider pure stopping times $\sigma\in[0,1]$ for the uninformed player (Remark \ref{rem-Laraki-Solan}). Let $\tau_1,\tau_2\in\te^R$ and let $\xi^1,\xi^2$ be the associated (deterministic) generating processes. Consider first the case in which $\xi^1_{\frac{1}{2}-}+\xi^2_{\frac{1}{2}-}>\delta$ for some $\delta\in(0,1)$ and fix $\ve \in (0, \frac12)$. For $\sigma=\frac{1}{2}-\ve$ we have
\begin{align*}
N((\tau_1,\tau_2),\sigma)
&=
\ee\big[\indd{\mcalI=1}(1\cdot\indd{\tau_1\les\sigma}+0\cdot\indd{\sigma<\tau_1}) + \indd{\mcalI=2}(1\cdot\indd{\tau_2\les\sigma}+0\cdot\indd{\sigma<\tau_2})\big]\\
&=  \frac12\big(\xi^1_{\frac{1}{2}-\ve} + \xi^2_{\frac{1}{2}-\ve}\big),
\end{align*}
thus implying
\begin{equation}\label{eq:last0}
\sup_{\sigma\in[0,1]} N((\tau_1,\tau_2),\sigma) \ges \lim_{\sigma\to \frac{1}{2}-} N((\tau_1,\tau_2),\sigma)=  \frac12(\xi^1_{\frac{1}{2}-}+\xi^2_{\frac{1}{2}-})>\frac{\delta}{2}>0.
\end{equation}
If, instead, $\xi^1_{\frac{1}{2}-}+\xi^2_{\frac{1}{2}-}\le\delta$ so that, in particular, $\xi^1_{\frac{1}{2}-}\vee\xi^2_{\frac{1}{2}-}\le\delta$, then
\begin{equation}\label{eq:last1}
\begin{aligned}
\sup_{\sigma\in[0,1]} N((\tau_1,\tau_2),\sigma)
&\ges 
N((\tau_1,\tau_2),1)\\
&\ge 
\ee\big[\indd{\mcalI=1}(1\cdot\indd{\tau_1<\frac{1}{2}}+3\cdot\indd{\tau_1\ge\frac{1}{2}}) + \indd{\mcalI=2}(-1)\big]\\
&\ges \frac{1}{2}\left(\xi^1_{\frac{1}{2}-}+3\big(1-\xi^1_{\frac{1}{2}-}\big)\right) - \frac{1}{2}
=1-\xi^1_{\frac{1}{2}-}\ge 1-\delta>0.
\end{aligned}
\end{equation}
Combining \eqref{eq:last0} and \eqref{eq:last1} we have $V^*>0$.
\end{proof}


\bibliographystyle{abbrvnat}
\bibliography{biblio}
\end{document}